\documentclass[oneside]{compositio}

\usepackage[theoremfont]{newtxtext}
\usepackage[varvw,smallerops]{newtxmath}

\usepackage{verbatim}

\usepackage{tikz-cd} 

\usepackage{array}
\newcolumntype{P}[1]{>{\raggedright\arraybackslash}p{#1}}
\usepackage{tabularx}
\usepackage{float}

\usepackage{hyperref}
\hypersetup{
    colorlinks=false, 
    linktoc=all,     
    linkcolor=black,  
}

\usepackage{placeins}

\newtheorem{thm}{Theorem}[section]
\newtheorem{cor}[thm]{Corollary}
\newtheorem{prop}[thm]{Proposition}

\theoremstyle{definition}
\newtheorem{defn}[thm]{Definition}

\theoremstyle{remark}
\newtheorem{exmp}[thm]{Example}
\newtheorem{rem}[thm]{Remark}

\newcommand{\Q}{\mathbb{Q}} 
\newcommand{\C}{\mathbb{C}} 
\renewcommand{\P}{\mathbb{P}}
\newcommand{\Z}{\mathbb{Z}}
\newcommand{\N}{\mathbb{N}}

\newcommand{\G}{\mathbb{G}}
\newcommand{\E}{\mathcal{E}}
\newcommand{\Oo}{\mathcal{O}}
\renewcommand{\L}{\mathbb{L}}
\newcommand{\Ll}{\mathcal{L}}
\newcommand{\M}{\mathcal{M}}
\newcommand{\Hh}{\mathcal{H}}
\newcommand{\oM}{{\mkern5mu\overline{\mkern-5mu\M\mkern-1.5mu}\mkern2mu}}
\newcommand{\F}{\mathcal{F}}

\DeclareMathOperator{\Spec}{Spec}

\begin{document}

\title{Localization, cobordism, and twisted $\chi_y$-genera of $\oM_{0,n}$}

\author{Benjamin Ellis-Bloor}

\email{benjaminellis-bloor2026@u.northwestern.edu}

\address{Department of Mathematics, Northwestern University, Evanston,
  IL 60208}


\begin{abstract}

We use the method of equivariant localization to obtain a recursion for the cobordism-valued one-point psi-class intersections on $\oM_{0,n}$, and hence the cobordism classes $[\oM_{0,n}]$. Restricting to the hyperbolic formal group law, we recover the Poincar\'{e} polynomials of $\oM_{0,n}$ as a special case, and obtain new formulas for Hirzebruch $\chi_y$-genera of $\oM_{0,n}$ twisted by tautological line bundles. 

\end{abstract}

\maketitle

\section{Introduction}

Intersection theory on the moduli space $\oM_{0,n}$ of genus zero stable curves with $n$ marked points can be extended naturally to cobordism. In this article, we continue the development of this extension begun in \cite{EllisBloor}, and its purpose is threefold, which we outline below. As in \cite{EllisBloor}, we work in the setting of algebraic cobordism $\Omega^*$ due to Levine--Morel \cite{AlgCob} and Levine--Pandharipande \cite{LevinePandharipande}, however our results also apply in the setting of complex cobordism. Recent work of Debnath--Zeng \cite{DebnathZeng} and Gandhi--Partida \cite{GandhiPartida} study the algebraic cobordism rings of $\oM_{0,n}$ and wonderful varieties of hyperplane arrangements (at the maximal building set), respectively. 

First, we give an alternative definition of the \textit{psi-classes} in cobordism, which simplifies the analogue of the string equation, introduced by Witten \cite{Witten} in ordinary cohomology and extended to cobordism in the author's previous work \cite{EllisBloor}. The string equation gives a recursion for all psi-class intersections in cobordism. This change of variables also yields a particularly simple analogue of the dilaton equation, also originally due to Witten \cite{Witten}. Furthermore, we give a closed formula for the indecomposable part of the cobordism class $[\oM_{0,n}]$. 

Second, we show that the string equation simplifies considerably if we restrict to the hyperbolic formal group law, which yields closed formulas for one-point psi-class intersections. We also explain how this can be used to recover formulas for the Poincar\'{e} polynomials of $\oM_{0,n}$ due to Keel \cite{Keel}, Manin \cite{Manin} and Getzler \cite{Getzler} as a special case, and we give new computations of certain twisted $\chi_y$-genera of $\oM_{0,n}$. 

Third, we prove an alternative recursion for the one-point psi-class intersections, and hence the cobordism classes $[\oM_{0,n}]$, using the equivariant localization formula in algebraic cobordism developed by Krishna \cite{Krishna1}, applied to the moduli space $\oM_{0,n}(\P^1,1)$ of genus zero stable maps to $\P^1$ of degree one. This then gives a second proof of the closed formulas obtained using the hyperbolic formal group law. 

\subsection{The string and dilaton equations in cobordism}

We will work with the usual definition of the Chern classes and pushforward maps in algebraic cobordism \cite{AlgCob,LevinePandharipande}, but we use a change of variables in our convention for the universal formal group law. For $X$ a smooth variety over an algebraically closed field $k$ of characteristic zero and a line bundle $L \to X$, we define
\begin{equation*}
\psi(L) := -c_{1}^{\Omega}(L^\vee) \in \Omega^1(X),
\end{equation*}
where $c_{1}^\Omega(L^\vee)$ is the first Chern class of $L^\vee$ in algebraic cobordism. Our convention for the universal formal group law $x+_{\Omega} y \in \Omega^{*}(\Spec k)[\![x,y]\!]$ is such that for two line bundles $L,M \to X$, we have
\begin{equation} \label{OurConvention}
\psi(L \otimes M) = \psi(L)+_{\Omega} \psi(M).
\end{equation}
We let $\overline{x} \in \Omega^*(\Spec k)[\![x]\!]$ be the inverse of the universal formal group law, defined as the unique power series satisfying $x+_{\Omega} \overline{x} = 0$. Since $\Omega^*$ is the universal oriented cohomology theory, there is an orientation map $\Omega^* \to \text{CH}^*$ to the Chow ring, and under this map $\psi(L)$ is sent to $c_{1}^{\text{CH}}(L)$, since
\begin{equation*}
-c_{1}^{\text{CH}}(L^\vee) = c_{1}^{\text{CH}}(L).
\end{equation*}
For $X$ projective, and $\alpha \in \Omega^*(X)$, we let $\int_{X}^{\Omega} \alpha := p_{*}(\alpha)$, where $p_{*} \colon \Omega^{*}(X) \to \Omega^{*-\dim X}(\Spec k)$ is the pushforward map induced by the projection $p \colon X \to \Spec k$. 

For $n \ge 3$ and $1 \le i \le n$, let $L_{i}$ denote the tautological line bundle on $\oM_{0,n}$ corresponding to the $i^{\text{th}}$ marked point. We define the corresponding psi-class in algebraic cobordism $\Omega^*$ to be
\begin{equation*}
\psi_i = \psi(L_i) \in \Omega^1(\oM_{0,n}).
\end{equation*}
This modified definition of the psi-classes is a generalization to cobordism of the modification in $K$-theory given in \cite[Section 5.2]{EllisBloor}. For integers $d_{1},\dots,d_{n} \ge 0$ with $|d| := d_{1}+\dots+d_{n}$, we define
\begin{equation*}
  \psi^d = \psi_1^{d_1} \dotsm \psi_n^{d_n} ,
\end{equation*}
and for each $1 \le i \le n$, we let
\begin{equation*}
  \psi^d/\psi_i =
  \begin{cases}
    \psi_1^{d_1} \dotsm \psi_i^{d_i-1} \dotsm \psi_n^{d_n} , & d_i>0 , \\
    0 , & d_i=0 .
  \end{cases}
\end{equation*}
For $T$ a subset of
$\{1,2,\dots,n\}$ with $\vert{T}\vert \ge 2$, $\vert{T^c}\vert \ge 2$ and
$\vert{T \cap \{1,2,3\}}\vert \le 1$, let
\begin{equation*}
i_{T} \colon D^{T} \to \oM_{0,n}
\end{equation*}
be the inclusion of the divisor
$D^{T} = \oM_{0,\vert{T}\vert+1} \times \oM_{0,\vert{T^c}\vert+1}$ whose
generic points are curves with a single node, and the two components
containing markings labeled by the elements of $T$ and $T^{c}$. We denote by $\bullet$ the node-branch marking in
either factor of the product, and write
$L_{a} = \text{pr}_{1}^{*}L_{\bullet}$ and
$L_{b} = \text{pr}_{2}^{*}L_{\bullet}$ for the pullbacks under the two
projections. 

The following is a simplification of the string equation in cobordism proved in previous work of the author \cite{EllisBloor}, using our modified definition of the psi-classes.
\begin{thm}[String equation \cite{EllisBloor}] \label{StringEquationIntro}
For $n \ge 3$ and $0 \le |d| \le n-2$, we have
\begin{multline*} 
    \int_{\oM_{0,n+1}}^\Omega 1\cdot\psi^d  \\
    \begin{aligned}
      &= [\P^{1}] \int_{\oM_{0,n}}^\Omega \psi^d
      +\sum_{i=1}^{n}\int_{\oM_{0,n}}^\Omega \psi^d/\psi_i+ \sum_{T}\int_{D^{T}}^\Omega \biggl(
      \frac{1}{\overline{\psi}_{a}\bigl(\psi_{a}+_{\Omega}\psi_{b}\bigr)}
      - \frac{1}{\bigl(\psi_{a}+_{\Omega}\psi_{b} \bigr) \psi_{b}} + \frac{1}{\psi_a\psi_b}
      \biggr) i_{T}^{*} \psi^d. \\
    \end{aligned}
\end{multline*}
\end{thm}
The first and third terms on the right-hand side in Theorem \ref{StringEquationIntro} do not appear in the string equation in the Chow ring. The first term appears first in $K$-theory (the multiplicative formal group law), and the third term appears first in the hyperbolic formal group law (see Section \ref{PoincareIntro}). 

It was remarked by Givental \cite{Givental} that there is no simple formula for the string equation in (rational) quantum cobordism theory, but that the dilaton equation holds in any genus. We prove geometrically in integral cobordism, using our modified definition of the psi-classes, that in genus zero the dilaton equation in cobordism has no extra terms compared to the ordinary dilaton equation, unlike with the string equation.

\begin{prop}[Dilaton equation] \label{DilatonEquationIntro}
For $n \ge 3$ and $0 \le |d| \le n-3$, we have
\begin{equation*}
\int_{\oM_{0,n+1}}^\Omega \psi_{n+1} \cdot \psi^d = (n-2)\int_{\oM_{0,n}}^\Omega \psi^d.
\end{equation*}
\end{prop}

Theorem \ref{StringEquationIntro} can in particular be used to calculate the integral cobordism classes $[\oM_{0,n}] \in \Omega^*(\Spec k)$, as proved in \cite{EllisBloor}. We know that $\Omega^*(\Spec k)$ is isomorphic to the Lazard ring $\L^* \cong \Z[u_{1},u_{2},u_{3},\dots]$ which classifies (one-dimensional and commutative) formal group laws, where $u_{i}$ has degree $-i$. Under the choice of integral generators given in \cite{EllisBloor}, we prove the following closed formula for the indecomposable part of the cobordism classes $[\oM_{0,n}]$.

\begin{prop} \label{IndecomposableIntro}
For $n \ge 2$, under the generator normalization of \cite{EllisBloor}, the coefficient of $u_{n-1}$ in $[\oM_{0,n+2}]$ is given by the closed formula
\begin{equation*}
\frac{(-1)^n}{\nu(n)}\left(\frac{1}{2}{2n \choose n}-1\right),
\end{equation*}
where
\begin{equation*}
  \nu(n) =
  \begin{cases}
    p , & \text{if $n = p^r$ for some prime $p$ and $r \in \N$} , \\
    1 , & \text{otherwise}.
  \end{cases}
\end{equation*}
\end{prop}

\subsection{Hyperbolic formal group law and twisted $\chi_y$-genera of $\oM_{0,n}$} \label{PoincareIntro}

For formal variables $\mu_1$ and $\mu_2$ of degrees $-1$ and $-2$, the \textit{hyperbolic formal group law} (see \cite{LZZ}), in our convention given by \eqref{OurConvention}, is the formal group law over $\Z[\mu_1,\mu_2]$ defined by
\begin{equation*}
x+_{H} y = \frac{x+y+\mu_1 xy}{1+\mu_2 xy}.
\end{equation*}
We will additionally adjoin formal variables $\alpha$ and $\beta$, both of degree $-1$, with the relations $\mu_1 = \alpha+\beta$ and $\mu_2 = -\alpha\beta$. The multiplicative formal group law (corresponding to connective $K$-theory) is obtained by setting $\mu_2 = 0$. Alternatively, setting $\mu_1 = 0$ gives the \textit{Lorentz formal group law},
\begin{equation*}
x+_{L}y = \frac{x+y}{1+\alpha^2 xy}. 
\end{equation*}
Importantly, the third term on the right-hand side of Theorem \ref{StringEquationIntro} simplifies to a constant when restricted to the hyperbolic formal group law, namely
\begin{equation} \label{HyperbolicSimplification}
\frac{1}{\overline{x}(x+_{H} y)}-\frac{1}{(x+_{H} y)y} +\frac{1}{xy} = -\mu_2.
\end{equation}
Using this, we prove the following closed formula for the one-point psi-class intersections in the hyperbolic formal group law. Form the exponential generating function
\begin{equation*}
\Hh(z,q) = \sum_{n \ge 3} \frac{q^{n-1}}{(n-1)!}\sum_{j = 0}^{n-3} z^{j} \int_{\oM_{0,n}}^{H} \psi_1^j.
\end{equation*}
For any parameter $\lambda$, define the series
\begin{equation*}
\Phi_{\lambda}(q) = \sum_{j \ge 1} \lambda^{j-1}\frac{q^j}{j!} = \frac{e^{\lambda q}-1}{\lambda},
\end{equation*}
where $\Phi_{0}(q) = q$, and let $\tau(q)$ be the compositional inverse of the series
\begin{equation*}
\frac{\alpha\Phi_{\beta}(q)-\beta\Phi_{\alpha}(q)}{\alpha-\beta}.
\end{equation*}
\begin{thm} \label{HyperbolicClosedFormulaIntro}
The generating function $\Hh(z,q)$ is given by
\begin{equation*}
\Hh(z,q) = \frac{z(\Phi_{z+\alpha+\beta}(\tau(q))-q)+\alpha\beta \Hh(0,q)}{(z+\alpha)(z+\beta)},
\end{equation*}
where 
\begin{equation*}
\Hh(0,q) = \frac{\Phi_{\alpha}(\tau(q))-\Phi_{\beta}(\tau(q))}{\alpha-\beta}.
\end{equation*}
In the Lorentz formal group law, we have
\begin{equation*}
\Hh(z,q) = \frac{e^{\frac{z}{\alpha}\textnormal{arcsinh}(\alpha q)}-\sqrt{1+\alpha^2q^2}-zq}{z^2-\alpha^2}.
\end{equation*}
\end{thm}
As an application, we derive an explicit closed formula for certain twists of the Hirzebruch $\chi_y$-genus \cite{Hirzebruch}. For $j \ge 0$, consider the $\chi_y$-genus of $\oM_{0,n}$ with coefficients in $L_{1}^j$, namely
\begin{equation} \label{TwistedChiyGenusIntro}
\chi_y(\oM_{0,n},L_1^j) := \sum_{p} \chi\bigl(\oM_{0,n},\Omega_{\oM_{0,n}}^{p} \otimes L_1^j\bigr)y^p,
\end{equation}
where here $\Omega_{\oM_{0,n}}^p$ denotes the sheaf of algebraic $p$-forms on $\oM_{0,n}$, and $\chi$ is the coherent-sheaf Euler characteristic. Form the exponential generating function
\begin{equation*}
G(z,q) = \sum_{n \ge 3} \frac{q^{n-1}}{(n-1)!}\sum_{j \ge 0} z^{j} \chi_y(\oM_{0,n},L_1^j),
\end{equation*}
and let $\tau(q)$ be the compositional inverse of the series
\begin{equation*}
\frac{\Phi_{-y}(q)+y\Phi_{1}(q)}{1+y}.
\end{equation*}
\begin{prop} \label{TwistedChiyM0nIntro}
The generating function $G(z,q)$ is given by the closed formula
\begin{equation*}
G(z,q) = \frac{\Phi_{\frac{1+yz}{1-z}}(\tau(q))-\Phi_{-y}(\tau(q))}{1+y}.
\end{equation*}
When $y = 1$, the generating function for the twisted signatures is given by
\begin{equation*}
G(z,q) = \frac{1}{2}\left(\frac{1-z}{1+z}(e^{\frac{1+z}{1-z}\textnormal{arcsinh}\,q}-1)+e^{-\textnormal{arcsinh}\,q}-1\right).
\end{equation*}
\end{prop}
\subsubsection{Poincar\'{e} polynomials of $\oM_{0,n}$}
We now work over $\C$. Setting $j = 0$ in \eqref{TwistedChiyGenusIntro} gives the ordinary Hirzebruch $\chi_y$-genus of $\oM_{0,n}$. By Keel \cite{Keel}, we know that the cycle map from the Chow ring to the cohomology of $\oM_{0,n}$ is an isomorphism, which implies that
\begin{equation*}
\chi_y(\oM_{0,n}) = \sum_{p} (-1)^p \dim_{\C} H^{2p}(\oM_{0,n},\C)y^p.
\end{equation*}
In particular, defining the Poincar\'{e} polynomial 
\begin{equation*}
\text{P}_{n}(z) = \sum_{p = 0}^{n-3}\dim_{\C} H^{2p}(\oM_{0,n},\C)z^p,
\end{equation*}
we have that $\text{P}_{n}(z) = \chi_{-z}(\oM_{0,n})$. The Hirzebruch $\chi_y$-genus is a ring homomorphism $\Omega^{*}(\Spec \C) \to \Z[y]$, and so it has an associated formal group law. In our convention given by \eqref{OurConvention}, it is given by
\begin{equation*}
u+_{\chi_y} v = \frac{u+v+(1-y)uv}{1+yuv},
\end{equation*}
which is hyperbolic. By Theorem \ref{StringEquationIntro} and the simplification given by \eqref{HyperbolicSimplification}, we recover the well-known quadratic recursion for the Poincar\'{e} polynomials due to Keel \cite{Keel} and Manin \cite{Manin},
\begin{equation*}
\text{P}_{n+1}(z) = (1+z)\text{P}_{n}(z)+z\sum_{\substack{i+j = n \\ i,j \ge 2}} {n-1 \choose i} \text{P}_{i+1}(z)\text{P}_{j+1}(z),
\end{equation*}
where $P_{3}(z) = 1$. Moreover, the generating function formulas in Theorem \ref{HyperbolicClosedFormulaIntro} can be thought of as a three-parameter generalization of the generating function formulas of Getzler \cite{Getzler} and Manin \cite{Manin} for the Poincar\'{e} polynomials. Note that Getzler more generally calculated the $S_n$-equivariant Poincar\'{e} polynomials, of which the formulas we consider are a special case. Recently, separate works by Aluffi--Marcolli--Nascimento \cite{AMN} and Eur--Ferroni--Matherne--Pagaria--Vecchi \cite{EFMPV} give closed formulas for the Poincar\'{e} polynomials of $\oM_{0,n}$, and Cotti \cite{Cotti} relates them to the earlier descriptions by Getzler and Manin. 

\subsection{Localization in Cobordism}

In Section \ref{LocalizationSection}, we give an alternative recursion for the one-point psi-class intersections in cobordism, and hence the classes $[\oM_{0,n}]$, using equivariant localization. In particular, we use the localization formula in algebraic cobordism due to Krishna \cite{Krishna1}, applied to a $\C^*$-action on the moduli space $\oM_{0,n}(\P^1,1)$ of genus zero stable maps to $\P^1$ of degree one. Namely, we apply the localization formula to calculate the pushforward
\begin{equation*}
\int_{\oM_{0,n}(\P^1,1)}^\Omega \text{ev}_{1}^*(h),
\end{equation*}
where $h = c_{1}^{\Omega}(\Oo_{\P^1}(1))$ is the hyperplane class and $\text{ev}_{1} \colon \oM_{0,n}(\P^1,1) \to \P^1$ is the evaluation map at the first marked point. After lifting this to a pushforward in equivariant cobordism, the recursion comes from the vanishing of all coefficients of negative powers of the equivariant parameter $t$ occuring in the expansion given by the localization formula. This can be packaged in terms of exponential generating functions as follows. Define
\begin{equation*}
E(t,q) = t-q+\sum_{n \ge 3}\frac{q^{n-1}}{(n-1)!} \int_{\oM_{0,n}}^\Omega \frac{1}{t+_{\Omega}\psi_1}.
\end{equation*}
\begin{thm} \label{LocalizationIntro}
The values of $\int_{\oM_{0,n}}^\Omega \psi_1^{j}$ for all $0 \le j \le n-3$ are uniquely determined by the condition that the coefficients of all negative powers of $t$ (coefficient-wise in $q$) in the expression
\begin{equation*}
\frac{\partial E(t,q)}{\partial q}\cdot \frac{E(\bar{t},q)}{\bar{t}}
\end{equation*}
are equal to zero, together with the fact that $\int_{\oM_{0,3}}^\Omega 1 = 1$. 
\end{thm}
We also use this recursion to give an alternative proof of Theorem \ref{HyperbolicClosedFormulaIntro}. In particular, this gives another method of recovering the known formulas for the Poincar\'{e} polynomials of $\oM_{0,n}$, as well as another proof of our formulas for the twisted $\chi_y$-genera in Proposition \ref{TwistedChiyM0nIntro}. 

\subsection{Acknowledgements}

We thank our advisor Ezra Getzler for helpful feedback and suggestions regarding this work.

\subsection{AI usage}

OpenAI's GPT-5.6 Sol and GPT-6 assisted in developing the closed formulas in Theorem \ref{HyperbolicClosedFormulaIntro}, Proposition \ref{TwistedChiyM0nIntro} and Proposition \ref{IndecomposableIntro}, and these calculations were independently verified and simplified by the author. GPT-6 was used for proofreading, and to explore proof strategies for the cobordism-valued dilaton equation in Proposition \ref{DilatonEquationIntro}, whose statement was provided by the author. The paper was entirely written by the author, who understands and is responsible for all statements inside of it.

\section{The string and dilaton equations, and the hyperbolic formal group law}

\subsection{The string equation in cobordism}

Let $k$ be an algebraically closed field of characteristic zero, which all our varieties will be defined over unless stated otherwise. We use the properties of algebraic cobordism $\Omega^*$ due to Levine--Morel \cite{AlgCob} and Levine--Pandharipande \cite{LevinePandharipande} as summarized in \cite[Section 2]{EllisBloor}, and use the usual pushforward maps and Chern classes. We identify $\Omega^*(\Spec k)$ with the Lazard ring $\L^*$ classifying (commutative and one-dimensional) formal group laws. For a smooth variety $X$ over $k$ and a line bundle $L \to X$, define
\begin{equation*}
\psi(L) := -c_{1}^{\Omega}(L^\vee) \in \Omega^1(X).
\end{equation*}
We then use the convention that $x+_{\Omega}y \in \L^*[\![x,y]\!]$ denotes the universal formal group law which satisfies
\begin{equation} \label{NewFGL}
\psi(L \otimes M) = \psi(L) +_{\Omega} \psi(M)
\end{equation}
for line bundles $L,M \to X$. Let $\overline{x} \in \L^*[\![x]\!]$ be the inverse of the universal formal group law, i.e. the unique power series satisfying $x+_{\Omega}\overline{x} = 0$, so that $\overline{\psi(L)} = -c_{1}^{\Omega}(L)$. 

It will also be convenient to define
\begin{equation} \label{UsualFGL}
x+_{\widetilde\Omega} y := -((-x)+_{\Omega} (-y)) \in \L^*[\![x,y]\!],
\end{equation}
which satisfies the usual convention that
\begin{equation*}
c_{1}^{\Omega}(L \otimes M) = c_{1}^{\Omega}(L)+_{\widetilde\Omega} c_{1}^\Omega(M).
\end{equation*}
We denote the inverse of this formal group law by $\iota(x) = -(\overline{-x})$, so $c_{1}^\Omega(L) = \iota(c_{1}^\Omega(L^\vee))$.

We begin by recalling the cobordism-valued string equation given in \cite[Theorem 4.1]{EllisBloor}, however we use a modified definition of the psi-classes. For each $1 \le i \le n$, let $L_{i}$ denote the tautological line bundle on $\oM_{0,n}$ corresponding to the $i^{\text{th}}$-marked point, and let 
\begin{equation} \label{NewCoordinates}
\psi_{i} = \psi(L_i) \in \Omega^1(\oM_{0,n}).
\end{equation}
For integers $d_{1},\dots,d_{n} \ge 0$ with $|d| := d_{1}+\dots+d_{n}$, we define
\begin{equation*}
  \psi^d = \psi_1^{d_1} \dotsm \psi_n^{d_n} ,
\end{equation*}
and for each $1 \le i \le n$, we let
\begin{equation*}
  \psi^d/\psi_i =
  \begin{cases}
    \psi_1^{d_1} \dotsm \psi_i^{d_i-1} \dotsm \psi_n^{d_n} , & d_i>0 , \\
    0 , & d_i=0 .
  \end{cases}
\end{equation*}

Furthermore, for each subset $T \subset \{1,2,\dots,n\}$ with $\vert{T}\vert \ge 2$, $\vert{T^c}\vert \ge 2$ and $\vert{T \cap \{1,2,3\}}\vert \le 1$, we denote by $D^{T}$ the divisor of $\oM_{0,n}$ isomorphic to $\oM_{0,\vert{T}\vert+1} \times \oM_{0,\vert{T^c}\vert+1}$, with inclusion $i_{T} \colon D^T \to \oM_{0,n}$. We let $\psi_a$ and $\psi_b$ be the pullbacks of the psi-classes at the node branch markings under the two projections from $\oM_{0,\vert{T}\vert+1} \times \oM_{0,\vert{T^c}\vert+1}$ to the first and second factors respectively. 

For $X$ a smooth projective variety over $k$ and $p \colon X \to \Spec k$ the projection, we use the notation $\int_X^\Omega \alpha := p_{*}(\alpha)$ for $\alpha \in \Omega^*(X)$, where $p_{*}$ is the pushforward map in $\Omega^*$ induced by $p$, which increases degree by $-\dim_k X$. More generally, the pushforward maps in $\Omega^*$ increase degree by the relative codimension. We now describe how the string equation of \cite[Theorem 4.1]{EllisBloor} transforms under the new definition of the psi-classes given by \eqref{NewCoordinates}. 
\begin{prop}[String equation] \label{NewStringEquation}
For $n \ge 3$ and $0 \le |d| \le n-2$, we have
\begin{multline*} 
    \int_{\oM_{0,n+1}}^\Omega 1\cdot\psi^d  \\
    \begin{aligned}
      &= [\P^{1}] \int_{\oM_{0,n}}^\Omega \psi^d
      +\sum_{i=1}^{n}\int_{\oM_{0,n}}^\Omega \psi^d/\psi_i+ \sum_{T}\int_{D^{T}}^\Omega \biggl(
      \frac{1}{\overline{\psi}_a\bigl(\psi_{a}+_{\Omega}\psi_{b}\bigr)}
     - \frac{1}{\bigl(\psi_{a}+_{\Omega}\psi_{b} \bigr) \psi_{b}} + \frac{1}{\psi_a\psi_b} 
      \biggr) i_{T}^{*} \psi^d. \\
    \end{aligned}
\end{multline*}
Together with the initial condition $\int_{\oM_{0,3}}^\Omega 1 = 1$, this uniquely determines all psi-class intersections. 
\end{prop}
\begin{proof}
Let $\pi \colon \oM_{0,n+1} \to \oM_{0,n}$ be the universal curve forgetting the $(n+1)^{\text{st}}$-marked point. Recall that for $1 \le i \le n$ we have the relation
\begin{equation*}
L_{i} = \pi^{*}L_{i} \otimes \Oo(D^{\{i,n+1\}}).
\end{equation*}
Defining $D_{i} = \psi(\Oo(D^{\{i,n+1\}}))$, it follows from \eqref{NewFGL} that
\begin{equation} \label{ForgetfulRelation}
\psi_i = \pi^{*}\psi_{i}+_{\Omega} D_{i}.
\end{equation}
As in the proof of \cite[Theorem 4.1]{EllisBloor}, we know that $D_{i}D_{j} = 0$ and $\psi_{i}D_{i} = 0$ for $i \neq j$. Hence, following the same algebraic manipulations in the proof of \cite[Theorem 4.1]{EllisBloor} and writing $x+_{\Omega} y = x+y+\sum_{j,\ell \ge 1} a_{j\ell}x^{j}y^{\ell}$, we get
\begin{equation*}
\int_{\oM_{0,n+1}}^\Omega 1 \cdot \psi^{d} = \int_{\oM_{0,n}}^\Omega \pi_{*}(1)\psi^d+\sum_{i=1}^{n} \int_{\oM_{0,n}}^\Omega
  \biggl(\pi_{*}(D_{i})+\sum_{j,\ell \ge 1}
  a_{j\ell}\psi_{i}^{j}\pi_{*}(D_{i}^{\ell})\biggr)\psi^d/\psi_i.
\end{equation*}
Let $f \colon D^{\{i,n+1\}} \to \oM_{0,n+1}$ be the inclusion, so that $c_{1}^\Omega(\Oo(D^{\{i,n+1\}})) = f_{*}(1)$, which we denote by $\widetilde{D}_{i}$. By definition, we have $D_{i} = \overline{-\widetilde{D}_{i}}$. Noting that $\overline{-x}$ is divisible by $x$, write $(\overline{-x})^{\ell} = xq_{\ell}(x)$. Then, by the projection formula, we get
\begin{equation*}
D_{i}^{\ell} = f_{*}(1)q_{\ell}(\widetilde{D}_{i}) = f_{*}(q_{\ell}(f^{*}\widetilde{D}_{i})).
\end{equation*}
However, we can identify $f$ with the section of $\pi$ corresponding to the $i^{\text{th}}$ marked point, so that $f^{*}\widetilde{D}_{i} = c_{1}^\Omega(L_{i}^{\vee})$. It follows that
\begin{equation*}
\pi_{*}(D_{i}^{\ell}) = \pi_{*}f_{*}(q_{\ell}(-\psi_{i})) = -\frac{(\overline{\psi}_i)^{\ell}}{\psi_{i}},
\end{equation*}
and hence
\begin{align*}
\int_{\oM_{0,n+1}}^\Omega 1 \cdot \psi^{d} &= \int_{\oM_{0,n}}^\Omega \pi_{*}(1)\psi^d-\sum_{i=1}^{n} \int_{\oM_{0,n}}^\Omega \frac{1}{\psi_{i}}\left(\overline{\psi}_{i}+\sum_{j,\ell \ge 1} a_{j\ell} \psi_{i}^{j}(\overline{\psi}_{i})^{\ell}\right)\psi^d/\psi_i \\
&= \int_{\oM_{0,n}}^\Omega \pi_{*}(1)\psi^d+\sum_{i=1}^{n} \int_{\oM_{0,n}}^\Omega \psi^d/\psi_i, \\
\end{align*}
with the second equality following from $\psi_{i}+_{\Omega} \overline{\psi}_i = 0$. We then use the calculation of $\pi_*(1)$ given by \cite[Proposition 3.6]{EllisBloor}. The fact that all psi-class intersections are uniquely determined by this equation follows since $\int_{\oM_{0,n}}^\Omega \psi^d = 0$ for $|d| > n-3$. 
\end{proof}
\begin{exmp} \label{ChowK-theoryClosed}
In the Chow ring we have that $-c_{1}^{\text{CH}}(L_{i}^\vee) = c_{1}^{\text{CH}}(L_{i})$, so \eqref{NewCoordinates} agrees with the usual definition of the psi-classes in this case, and for $|d| = n-3$,
\begin{equation*}
\int_{\oM_{0,n}}^{\text{CH}} \psi^d = {n-3 \choose d_1,\dots,d_n}.
\end{equation*}
However, consider $K$-theory, namely the oriented cohomology theory $K^{0}(X)[\beta,\beta^{-1}] := K^{0}(X) \otimes_{\Z} \Z[\beta,\beta^{-1}]$ where $K^{0}(X)$ is the Grothendieck group of locally free coherent sheaves on $X$ and $\beta$ is a formal variable of degree $-1$. Then 
\begin{equation*}
\psi_i = ([\Ll_{i}]-1)\beta^{-1},
\end{equation*}
where $\Ll_i$ is the sheaf of sections of $L_{i}$, which is precisely the twist of the $K$-theoretic psi-classes considered in \cite[Section 5.2]{EllisBloor} leading to the closed formula
\begin{equation} \label{KClosedFormula}
\int_{\oM_{0,n}}^K \psi^d = \beta^{n-3-|d|} \binom{n-3}{n-3-|d|,d_{1},\dots,d_{n}}.
\end{equation}
\end{exmp}

Proposition \ref{NewStringEquation} in particular gives an inductive formula for the classes of $\oM_{0,n}$ as elements in the Lazard ring $\L^{*}$, since
\begin{equation*}
  [\oM_{0,n}] = \int_{\oM_{0,n}}^\Omega 1.
\end{equation*}
Identify $\L^{*} \cong \Z[u_{1},u_{2},u_{3},\dots]$ with $\vert{u_i}\vert = -i$, with the choice of integral generators corresponding to the presentation of the universal formal group law given in \cite{Hazewinkel} and \cite[Appendix A]{EllisBloor}. In terms of these generators, we have in low dimensions \cite[Example 4.3]{EllisBloor} that
\begin{itemize}
\item[] $[\oM_{0,4}] = u_{1}$
\item[] $[\oM_{0,5}] = 4u_{1}^{2}-3u_{2}$
\item[] $[\oM_{0,6}] = 31u_{1}^{3}-30u_{1}u_{2}+17u_{3}$
\item[] $[\oM_{0,7}] = 273u_{1}^{4}-317u_{1}^{2}u_{2}+70u_{2}^{2}+214u_{1}u_{3}-25u_{4}$
\item[] $[\oM_{0,8}] = 2898u_{1}^{5}-4063u_{1}^{3}u_{2}+2012u_{1}u_{2}^{2}+2765u_{1}^{2}u_{3}-1204u_{2}u_{3}-385u_{1}u_{4}+461u_{5}$.
\end{itemize}
We now aim to give a closed formula for the sequence $1,-3,17,-25,461,\dots$ of the coefficients of the $u_{n-3}$ in $[\oM_{0,n}]$ for $n > 3$. 
\begin{defn}
Let $I \subset \L^{*}$ be the ideal generated by the negatively graded elements. Then, the \textit{indecomposable quotient} $\mathcal{Q}\L^{*}$ is defined by 
\begin{equation*}
\mathcal{Q}\L^* = I/I^2.
\end{equation*}
\end{defn}
Note in particular that $\mathcal{Q}\L^{-n} \cong \Z$ for each $n > 0$, generated by $u_n$. Following \cite[Section 4.4.4]{AlgCob}, for a smooth proper variety $X$ of dimension $n$ with tangent bundle $T_X$ having ordinary Chern roots $\lambda_1,\dots,\lambda_n$, define the Milnor number
\begin{equation*}
S_{n}(X) := \int_{X}^{\text{CH}} \lambda_1^n+\dots+\lambda_n^n = n!\int_{X}^{\text{CH}} \text{ch}_{n}(T_X) \in \Z,
\end{equation*}
where $\text{ch}_n(T_X)$ is the degree $n$ component of the Chern character of $T_X$. Observe that for $[X],[Y] \in \L^*$ with $\dim X, \dim Y > 0$ and $\dim X+\dim Y = n$, we have that $S_n(X \times Y) = 0$, since $T_{X \times Y} \cong \text{pr}_1^*T_X \oplus \text{pr}_2^*T_Y$. Therefore, it follows that $S_n$ descends uniquely to a map $Q\L^{-n} \to \Z$, which we also denote by $S_n$. Next, following \cite[Section 1.6]{Hazewinkel}, define
\begin{equation*}
  \nu(n) =
  \begin{cases}
    p , & \text{if $n = p^r$ for some prime $p$ and $r \in \N$} , \\
    1 , & \text{otherwise}.
  \end{cases}
\end{equation*}
Furthermore, let $\ell_{\widetilde\Omega}(x) \in (\L^* \otimes_\Z \Q)[\![x]\!]$ denote the logarithm of the universal formal group law $x+_{\widetilde\Omega} y$, i.e. the unique power series with leading term $x$ such that
\begin{equation} \label{Logarithm}
\ell_{\widetilde\Omega}(x+_{\widetilde\Omega}y) = \ell_{\widetilde\Omega}(x)+\ell_{\widetilde\Omega}(y).
\end{equation}
In the convention for the universal formal group law corresponding to our choice of integral generators, we have \cite{Hazewinkel} 
\begin{equation*}
\ell_{\widetilde\Omega}(x) \equiv x+\sum_{n \ge 1} \frac{u_{n}}{\nu(n+1)}x^{n+1} \mod I^2. 
\end{equation*}
However, Mishchenko's formula (see \cite{Quillen}) says that
\begin{align*}
\ell_{\widetilde\Omega}(x) = \sum_{n \ge 0} [\P^n]\frac{x^{n+1}}{n+1},
\end{align*}
which implies that
\begin{equation} \label{Snfirst}
\frac{S_{n}(u_n)}{\nu(n+1)} = \frac{S_n(\P^n)}{n+1}. \\
\end{equation}
The Euler sequence
\begin{equation*}
0 \to \Oo_{\P^n} \to \Oo_{\P^n}(1)^{\oplus(n+1)} \to T_{\P^n} \to 0
\end{equation*}
gives that $S_{n}(\P^n) = n+1$, so \eqref{Snfirst} implies $S_{n}(u_n) = \nu(n+1)$. Therefore, given an $[X] \in \L^{-n}$, we have that
\begin{equation} \label{Sn}
[X] \equiv \frac{S_{n}(X)}{\nu(n+1)} \cdot u_n \mod I^2.
\end{equation}
\begin{prop}
For each $n \ge 2$, we have
\begin{equation*}
[\oM_{0,n+2}] \equiv \frac{(-1)^n}{\nu(n)}\left(\frac{1}{2}{2n \choose n}-1\right)u_{n-1} \mod I^2.
\end{equation*}
\end{prop}
\begin{proof}
Following \cite[Appendix B]{Eberhardt} and restricting to the case of genus zero, we have that
\begin{equation*}
\text{ch}(T_{\oM_{0,n+2}}^{\vee}) = n-1+\sum_{j \ge 1} \frac{\kappa_j-\Delta_j}{j!},
\end{equation*}
where $\kappa_j = \pi_{*}(\psi_{n+3}^{j+1})$ with $\pi \colon \oM_{0,n+3} \to \oM_{0,n+2}$ forgetting the $(n+3)^{\text{rd}}$ marked point, and 
\begin{equation*}
\Delta_j = \frac{1}{2} \sum_{\substack{T \subset \{1,\dots,n+2\} \\ \vert{T}\vert,\vert{T^c}\vert \ge 2}} (i_T)_{*}((\psi_a+\psi_b)^{j-1}).
\end{equation*}
Note here that the pushforwards and psi-classes are in the Chow ring, and that in this sum we do not impose the condition $\vert{T \cap \{1,2,3\}}\vert \le 1$ as we did in Proposition \ref{NewStringEquation}. Hence, we have that
\begin{equation} \label{SdM0n}
(-1)^{n-1}S_{n-1}(\oM_{0,n+2}) = \int_{\oM_{0,n+2}}^{\text{CH}} (\pi_{*}(\psi_{n+3}^{n})-\Delta_{n-1}) = 1-\int_{\oM_{0,n+2}}^{\text{CH}}\Delta_{n-1},
\end{equation} 
with the second equality using Example \ref{ChowK-theoryClosed}. Now, we have
\begin{align*}
\int_{\oM_{0,n+2}}^{\text{CH}} \Delta_{n-1} &= \frac{1}{2}\sum_{k=2}^{n}{n+2 \choose k} \int_{\oM_{0,k+1} \times \oM_{0,n-k+3}}^{\text{CH}} (\psi_a+\psi_b)^{n-2} \\
&= \frac{1}{2}\sum_{k=2}^{n} {n+2 \choose k} {n-2 \choose k-2} \left(\int_{\oM_{0,k+1}}^{\text{CH}} \psi_{\bullet}^{k-2}\right)\left(\int_{\oM_{0,n-k+3}}^{\text{CH}} \psi_{\bullet}^{n-k}\right) \\
&= \frac{1}{2} \sum_{k=2}^{n} {n+2 \choose k} {n-2 \choose n-k} \\
&= \frac{1}{2} {2n \choose n}. \\
\end{align*}
The result then follows by combining \eqref{Sn} and \eqref{SdM0n}.
\end{proof}

\subsection{The dilaton equation in cobordism}

In addition to the simplification of the string equation in cobordism as seen in Proposition \ref{NewStringEquation}, our modified definition of the psi-classes given by \eqref{NewCoordinates} yields a particularly simple dilaton equation in cobordism in genus zero, which has the identical form of the usual genus zero dilaton equation in the Chow ring \cite{Witten}. See \cite{Lee2} for the dilaton equation in $K$-theory in genus $g$. 

\begin{prop}[Dilaton equation]
For $n \ge 3$ and $0 \le |d| \le n-3$, we have
\begin{equation*}
\int_{\oM_{0,n+1}}^\Omega \psi_{n+1} \cdot \psi^d = (n-2)\int_{\oM_{0,n}}^\Omega \psi^d.
\end{equation*}
\end{prop}
\begin{proof}
As in the proof of Proposition \ref{NewStringEquation}, let $\pi \colon \oM_{0,n+1} \to \oM_{0,n}$ forget the $(n+1)^{\text{st}}$-marked point, and let $D_{i} = \psi(\Oo(D^{\{i,n+1\}}))$. Recall that $D_{i}D_{j} = 0$ and $\psi_i D_{i} = 0$ for $i \neq j$. We also have that $\psi_{n+1}D_{i} = 0$ for all $i$, since the corresponding tautological line bundle $L_{n+1}$ is trivial when restricted to each divisor $D^{\{i,n+1\}}$. Using \eqref{ForgetfulRelation}, we therefore get that
\begin{align*}
\int_{\oM_{0,n+1}}^\Omega \psi_{n+1}\cdot \psi^d &= \int_{\oM_{0,n+1}}^\Omega \psi_{n+1} \cdot \prod_{i=1}^{n} (\pi^{*}\psi_{i}+_{\Omega} D_{i})^{d_i} \\
&= \int_{\oM_{0,n}}^\Omega \pi_{*}(\psi_{n+1}) \cdot \psi^d. \\
\end{align*}
Our goal now is to prove that $\pi_{*}(\psi_{n+1}) = (n-2) \cdot 1$. To do this, first recall that
\begin{equation} \label{Ln+1}
L_{n+1} = \omega_{\oM_{0,n+1}/\oM_{0,n}}\biggl(\sum_{i=1}^{n} D^{\{i,n+1\}}\biggr).
\end{equation}
Letting $\widetilde{D}_{i} = c_{1}^\Omega(\Oo(D^{\{i,n+1\}}))$, we have that $\widetilde{D}_{i}\widetilde{D}_{j} = 0$ for $i \neq j$, and $c_{1}^\Omega(L_{n+1}^{\vee})\widetilde{D}_{i} = 0$. Hence, taking duals in \eqref{Ln+1}, we get
\begin{equation*}
c_{1}^\Omega(\omega_{\oM_{0,n+1}/\oM_{0,n}}^\vee) = c_{1}^\Omega(L_{n+1}^\vee)+\sum_{i=1}^{n} \widetilde{D}_{i}.
\end{equation*}
As in the proof of Proposition \ref{NewStringEquation}, we have $\pi_{*}(\widetilde{D}_{i}) = 1$, and so we get
\begin{equation*}
\pi_{*}(\psi_{n+1}) = n \cdot 1-\pi_{*}c_{1}^\Omega(\omega_{\oM_{0,n+1}/\oM_{0,n}}^\vee).
\end{equation*}
To calculate the pushforward on the right-hand side, we use Keel's construction of $\oM_{0,n+1}$ as an iterated blow-up, as in the proof of the string equation in cobordism \cite{EllisBloor}, and we will briefly recall here the notation used in \cite[Section 3]{EllisBloor}. 

Let $B_{1} = \oM_{0,n} \times \P^1$, and embed each divisor $D^T$ as a codimension two subvariety of $B_{1}$ via the composite $(\pi,\pi_{1,2,3,n+1}) \circ \sigma^{i}$ for $i \in T$, where $\sigma^i \colon \oM_{0,n} \to \oM_{0,n+1}$ is the section corresponding to the $i^{\text{th}}$-marked point and $\pi_{1,2,3,n+1}$ forgets all but the markings labeled by $1,2,3,n+1$. Then, for each $k \ge 1$, let $f_{k+1} \colon B_{k+1} \to B_{k}$ be the blow-up along the strict transforms of the $D^T$ such that $\vert{T^c}\vert = k+1$ (which are disjoint), and let $g_{k+1}$ be the composite $B_{k+1} \to B_{1}$. By Keel \cite{Keel}, we have that $\oM_{0,n+1}$ is isomorphic to $B_{n-2}$. 

Consider
\begin{equation*}
\omega_{B_{k+1}/\oM_{0,n}} = \omega_{B_{k+1}} \otimes g_{k+1}^*p_{1}^*\omega_{\oM_{0,n}}^{\vee},
\end{equation*}
where $p_{1} \colon B_{1} \to \oM_{0,n}$ is the projection. For each $\vert{T^c}\vert = k+1$, let $E_{k+1}^T$ denote the corresponding exceptional divisor of the blow-up $f_{k+1} \colon B_{k+1} \to B_{k}$. Since $f_{k+1}$ is a blow-up along codimension two centers, we have
\begin{equation*}
\omega_{B_{k+1}} = f_{k+1}^*\omega_{B_k} \otimes \Oo\biggl(\sum_{\vert{T^c}\vert = k+1} E_{k+1}^T\biggr),
\end{equation*}
and hence
\begin{equation} \label{OmegaDual}
\omega_{B_{k+1}/\oM_{0,n}}^{\vee} = f_{k+1}^*\omega_{B_k/\oM_{0,n}}^{\vee}\otimes\Oo\biggl(-\sum_{\vert{T^c}\vert = k+1} E_{k+1}^T\biggr).
\end{equation}
For ease of notation, let
\begin{equation*}
\alpha_{k} = c_{1}^\Omega(\omega_{B_k/\oM_{0,n}}^{\vee}) \hspace{0.5cm} \text{and} \hspace{0.5cm} e_{T} = c_{1}^\Omega(\Oo(E_{k+1}^T)).
\end{equation*}
We first prove that $(f_{k+1})_{*}\alpha_{k+1} = \alpha_k$. By \eqref{OmegaDual}, and using that the exceptional divisors $E_{k+1}^T$ for $\vert{T^c}\vert = k+1$ are disjoint, we have
\begin{align*}
\alpha_{k+1} &= f_{k+1}^{*}\alpha_{k}+_{\widetilde\Omega} \biggl(\sum_{\vert{T^c}\vert = k+1} \iota({e}_{T})\biggr) \\
&= f_{k+1}^*\alpha_{k}+\sum_{\vert{T^c}\vert = k+1} ((f_{k+1}^{*}\alpha_k+_{\widetilde\Omega}\iota({e}_{T}))-f_{k+1}^{*}\alpha_k). \\
\end{align*}
We now push forward by $f_{k+1}$. For $\vert{T^c}\vert = k+1$, let $i_{k,T} \colon D^T \to B_{k}$ be the inclusion. By the projection formula, and the calculation in the proof of \cite[Proposition 3.6]{EllisBloor}, we have
\begin{align*}
(f_{k+1})_{*}f_{k+1}^{*}\alpha_{k} &= \alpha_{k}\cdot(f_{k+1})_{*}(1) \\
&= \alpha_{k}\biggl(1+\sum_{\vert{T^c}\vert = k+1} (i_{k,T})_{*}\biggl(\frac{1}{\overline{\psi}_a\bigl(\psi_{a}+_{\Omega}\psi_{b}\bigr)} + \frac{1}{\psi_a\psi_b} - \frac{1}{\bigl(\psi_{a}+_{\Omega}\psi_{b} \bigr) \psi_{b}}\biggr)\biggr) \\
&= \alpha_{k}+\sum_{\vert{T^c}\vert = k+1} (i_{k,T})_{*} \biggl((i_{k,T}^{*}\alpha_k)\biggl(\frac{1}{\overline{\psi}_a\bigl(\psi_{a}+_{\Omega}\psi_{b}\bigr)}
     + \frac{1}{\psi_a\psi_b} - \frac{1}{\bigl(\psi_{a}+_{\Omega}\psi_{b} \bigr) \psi_{b}}\biggr)\biggr). \\
\end{align*}
To calculate $i_{k,T}^{*}\alpha_k$, the adjunction formula gives that
\begin{equation*}
i_{k,T}^{*}\omega_{B_k} = \omega_{D^T} \otimes \det N_{D^T/B_{k}}^\vee \hspace{0.5cm} \text{and} \hspace{0.5cm} i_{T}^*\omega_{\oM_{0,n}} = \omega_{D^T} \otimes N_{D^T/\oM_{0,n}}^\vee,
\end{equation*}
where $i_T \colon D^T \to \oM_{0,n}$ is the inclusion, satisfying $i_{T} = p_{1} \circ g_{k} \circ i_{k,T}$. By \cite[Lemma 3.4]{EllisBloor}, we have $N_{D^T/B_{k}} = (L_{a}^\vee \otimes L_{b}^\vee) \oplus L_{b}^\vee$. Then, since $N_{D^T/\oM_{0,n}} = L_{a}^{\vee} \otimes L_{b}^{\vee}$, it follows that
\begin{equation*}
i_{k,T}^{*}\omega_{B_k/\oM_{0,n}} = i_{k,T}^{*}\omega_{B_k} \otimes i_{T}^{*}\omega_{\oM_{0,n}}^\vee = L_{b},
\end{equation*}
and hence $i_{k,T}^*\alpha_k = -\psi_b$. That is, we have
\begin{equation} \label{FirstTerm}
(f_{k+1})_{*}f_{k+1}^{*}\alpha_{k} = \alpha_k+\sum_{\vert{T^c}\vert = k+1} (i_{k,T})_{*} \biggl(-\frac{\psi_b}{\overline{\psi}_a\bigl(\psi_{a}+_{\Omega}\psi_{b}\bigr)}
     - \frac{1}{\psi_a} + \frac{1}{\psi_{a}+_{\Omega}\psi_{b} }\biggr). 
\end{equation}
Next, for $\vert{T^c}\vert = k+1$, let $j_{k,T} \colon E_{k+1}^T \to B_{k+1}$ be the inclusion, so that $e_{T} = (j_{k,T})_{*}(1)$. Note also that $j_{k,T}^{*}e_{T} = \iota(\xi)$ where $\xi = c_{1}^\Omega(\Oo(1))$, identifying $E_{k+1}^T$ with $\P_{D^T}(N_{D^T/B_{k}})$. Then, using that $(f_{k+1}^{*}\alpha_k+_{\widetilde\Omega}\iota({e}_{T}))-f_{k+1}^{*}\alpha_k$ is divisible by $e_{T}$ together with the projection formula and the fact that $f_{k+1} \circ j_{k,T} = i_{k,T} \circ p$ where $p \colon \P_{D^T}(N_{D^T/B_{k}}) \to D^T$ is the projection, we have
\begin{align*}
(f_{k+1})_{*}((f_{k+1}^{*}\alpha_k+_{\widetilde\Omega}\iota({e}_{T}))-f_{k+1}^{*}\alpha_k) &= (f_{k+1})_{*}\biggl((j_{k,T})_{*}(1) \cdot \frac{(f_{k+1}^{*}\alpha_k+_{\widetilde\Omega}\iota({e}_{T}))-f_{k+1}^{*}\alpha_k}{e_{T}}\biggr) \\
&= (i_{k,T})_{*}p_{*}\biggl(\frac{(p^{*}i_{k,T}^{*}\alpha_k+_{\widetilde\Omega}\xi)-p^{*}i_{k,T}^*\alpha_k}{\iota(\xi)}\biggr) \\
&= (i_{k,T})_{*}p_{*}\biggl(\frac{((-p^{*}\psi_b)+_{\widetilde\Omega}\xi)+p^{*}\psi_b}{\iota(\xi)}\biggr). \\
\end{align*}
Using the fact as above that $N_{D^T/B_{k}} = (L_{a}^\vee \otimes L_{b}^\vee) \oplus L_{b}^\vee$, it follows from \cite[Theorem 5.30]{Vishik} that
\begin{equation*}
p_{*}\biggl(\frac{((-p^{*}\psi_b)+_{\widetilde\Omega}\xi)+p^{*}\psi_b}{\iota(\xi)}\biggr) = \frac{\psi_b}{\overline{\psi}_a\bigl(\psi_{a}+_{\Omega}\psi_{b}\bigr)}
     + \frac{1}{\psi_a} - \frac{1}{\psi_{a}+_{\Omega}\psi_{b} }.
\end{equation*}
See the proof of \cite[Proposition 3.6]{EllisBloor} for a similar calculation. Together with \eqref{FirstTerm}, we have now shown that
\begin{align*}
(f_{k+1})_{*}\alpha_{k+1} &= (f_{k+1})_{*}f_{k+1}^{*}\alpha_{k}+\sum_{\vert{T^c}\vert = k+1}(f_{k+1})_{*}((f_{k+1}^{*}\alpha_k+_{\widetilde\Omega}\iota({e}_{T}))-f_{k+1}^{*}\alpha_k) \\
&= \alpha_k. 
\end{align*}
By induction, we get that
\begin{equation*}
\pi_{*}c_{1}^\Omega(\omega_{\oM_{0,n+1}/\oM_{0,n}}^\vee) = (p_{1})_{*}c_{1}^\Omega(\omega_{B_{1}/\oM_{0,n}}^\vee) = (p_{1})_{*}c_{1}^\Omega(\text{pr}_{2}^*\Oo_{\P^1}(2)) = 2 \cdot 1,
\end{equation*}
where $\text{pr}_{2} \colon B_1 \to \P^1$ is the projection.
\end{proof}

\subsection{The hyperbolic formal group law} \label{HyperbolicFGLSection}

Let $\mu_1$ and $\mu_2$ be formal variables of degree $-1$ and $-2$ respectively. Then, the \textit{hyperbolic formal group law} in the convention of \eqref{NewFGL} is the formal group law over $\Z[\mu_1,\mu_2]$ defined by
\begin{equation} \label{HyperbolicFGL}
x+_{H} y = \frac{x+y+\mu_1 xy}{1+\mu_2 xy}.
\end{equation}
Note that \eqref{HyperbolicFGL} is precisely the formal group law associated to the singular elliptic curve given in Weierstrass coordinates by
\begin{equation*}
y^2-\mu_1 xy = x^3+\mu_2 x^2,
\end{equation*}
see \cite{LZZ}. If we additionally adjoin variables $\alpha$ and $\beta$ of degree $-1$ satisfying $\mu_1 = \alpha+\beta$ and $\mu_2 = -\alpha\beta$, then \eqref{HyperbolicFGL} in our convention is the formal group law associated to Hirzebruch's 2-parameter Todd genus \cite{Hirzebruch}. In particular, we get the Todd genus when $\mu_2 = 0$, the signature when $\mu_1 = 0$, and the Euler characteristic when $\alpha = \beta$. In the case of $\mu_1 = 0$, i.e. $\alpha = -\beta$, the formal group law
\begin{equation} \label{LorentzFGL}
x+_{L} y := \frac{x+y}{1+\alpha^2 xy}
\end{equation}
is called the \textit{Lorentz formal group law}. We now consider the restriction of Proposition \ref{NewStringEquation} to the hyperbolic formal group law. Form the exponential generating function
\begin{equation} \label{GenFun}
F = \sum_{n \ge 3} \frac{1}{n!} \sum_{d} t_{d} \int_{\oM_{0,n}}^H \psi^d,
\end{equation}
where $t_d = t_{d_1}\dotsm t_{d_n}$. 
\begin{prop} \label{StringHyperbolic}
In the oriented cohomology theory corresponding to the hyperbolic formal group law, we have that
\begin{equation*}
\frac{\partial F}{\partial t_{0}} = \frac{t_0^2}{2}+\mu_{1} F+\sum_{k \ge 0} t_{k+1}\frac{\partial F}{\partial t_k}-\frac{\mu_2}{2}\left(\frac{\partial F}{\partial t_0}\right)^2.
\end{equation*}
Together with the initial condition $F \vert_{t_0 = 0} = 0$, this differential equation uniquely determines $F$.
\end{prop}
\begin{proof}
The inverse of the hyperbolic formal group law is given by
\begin{equation} \label{InverseHyperbolicFGL}
\overline{x} = \frac{-x}{1+\mu_1 x}.
\end{equation}
A direct calculation then gives that
\begin{equation*}
\frac{1}{\overline{x}\bigl(x+_{H}y\bigr)}  
     - \frac{1}{\bigl(x+_{H}y \bigr) y} + \frac{1}{xy} = -\mu_2.
\end{equation*}
Furthermore, since the coefficient of $xy$ in the universal formal group law $x+_{\Omega} y$ is $[\P^1]$, the string equation becomes
\begin{equation*}
\int_{\oM_{0,n+1}}^H 1\cdot\psi^d = \mu_1 \int_{\oM_{0,n}}^H \psi^d
      +\sum_{i=1}^{n}\int_{\oM_{0,n}}^H \psi^d/\psi_i
      -\mu_2 \sum_T \biggl(\int_{\oM_{0,\vert{T}\vert+1}}^H \prod_{\ell \in T} \psi_\ell^{d_{\ell}}\biggr)\biggl(\int_{\oM_{0,\vert{T^c}\vert+1}}^H \prod_{\ell \in T^c} \psi_\ell^{d_{\ell}}\biggr).
\end{equation*}
The result follows by summing over $n$ using the definition in \eqref{GenFun} and the fact that $\int_{\oM_{0,3}}^H 1 = 1$. 
\end{proof}
We now aim to give a closed formula in the case $d_2 = \dots = d_n = 0$, i.e. we will look at the one-point pushforwards $\int_{\oM_{0,n}}^H \psi_1^{d_1}$. Define
\begin{equation*} 
\Hh(z,q) = \sum_{n \ge 3} \frac{q^{n-1}}{(n-1)!}\sum_{j=0}^{n-3} z^{j} \int_{\oM_{0,n}}^H \psi_1^{j}.
\end{equation*}
\begin{cor} \label{HyperbolicDE}
The generating function $\Hh(z,q)$ is uniquely determined by the differential equation
\begin{equation*}
(1+\mu_2 \Hh(0,q))\frac{\partial \Hh(z,q)}{\partial q} = (z+\mu_1)\Hh+q,
\end{equation*}
together with the initial condition $\Hh(z,0) = 0$. 
\end{cor}
\begin{proof}
Define the operator
\begin{equation*}
D_z = \sum_{j \ge 0} z^j \frac{\partial}{\partial t_j}.
\end{equation*}
We have that
\begin{equation*}
\Hh(z,q) = D_{z}F \vert_{t_0 = q, \text{   } t_i = 0 \text{ for $i > 0$}}.
\end{equation*}
The result then follows directly by applying $D_z$ to both sides of the string equation in Proposition \ref{StringHyperbolic} (and substituting $t_0 = q$ and $t_i = 0$ for $i > 0$), and rearranging terms, noting for example that
\begin{equation*}
D_{z} \biggl(\sum_{k \ge 0} t_{k+1}\frac{\partial F}{\partial t_k}\biggr) \bigg\vert_{t_0 = q, \text{   } t_i = 0 \text{ for $i > 0$}} = z\Hh(z,q).
\end{equation*}
\end{proof}
We can in fact go further to produce a closed formula for $\Hh(z,q)$. For any parameter $\lambda$ define the series
\begin{equation*}
\Phi_{\lambda}(q) = \sum_{j \ge 1} \lambda^{j-1}\frac{q^j}{j!} = \frac{e^{\lambda q}-1}{\lambda},
\end{equation*}
where in the second equality we understand that $\Phi_{0}(q) = q$. Recalling that $\mu_1 = \alpha+\beta$ and $\mu_2 = -\alpha\beta$, let $\tau(q)$ denote the compositional inverse of the power series
\begin{equation} \label{DefOmega}
\omega(q) = \frac{\alpha\Phi_\beta(q)-\beta\Phi_\alpha(q)}{\alpha-\beta}.
\end{equation}
\begin{thm} \label{HyperbolicClosedFormula}
The generating function $\Hh(z,q)$ in the hyperbolic formal group law is given by
\begin{equation*}
\Hh(z,q) = \frac{z(\Phi_{z+\alpha+\beta}(\tau(q))-q)+\alpha\beta \Hh(0,q)}{(z+\alpha)(z+\beta)} = \frac{z(\Phi_{z+\mu_1}(\tau(q))-q)-\mu_2 \Hh(0,q)}{z^{2}+\mu_1 z-\mu_2},
\end{equation*}
where 
\begin{equation*}
\Hh(0,q) = \frac{\Phi_{\alpha}(\tau(q))-\Phi_{\beta}(\tau(q))}{\alpha-\beta}.
\end{equation*}
\end{thm}
\begin{proof}
Let 
\begin{equation*}
\nu(q) = \frac{\Phi_{\alpha}(q)-\Phi_{\beta}(q)}{\alpha-\beta}.
\end{equation*}
Since for any fixed parameter $\lambda$ we have ${\Phi'_{\lambda}}(q) = 1+\lambda \Phi_{\lambda}(q)$, it follows that
\begin{equation} \label{domegadnu}
\frac{d\omega}{dq} = 1-\alpha\beta\nu \hspace{0.5cm} \text{    and    } \hspace{0.5cm} \frac{d\nu}{dq} = \omega+(\alpha+\beta)\nu.
\end{equation}
Now, let $\rho = \nu \circ \tau$. Since $\tau$ is by definition the compositional inverse of $\omega$, we get that
\begin{align*}
\frac{d\rho}{dq} &= (\omega(\tau(q))+(\alpha+\beta)\nu(\tau(q)))\frac{d\tau}{dq} \\
&= (q+(\alpha+\beta)\rho)\frac{d\tau}{dq} \\
&= \frac{q+(\alpha+\beta)\rho}{1-\alpha\beta\rho}, \\
\end{align*}
with the final equality following from differentiating both sides of the equality $\omega(\tau(q)) = q$. Re-arranging, we get
\begin{equation*}
(1-\alpha\beta\rho)\frac{d\rho}{dq} = (\alpha+\beta)\rho+q.
\end{equation*}
This is precisely the differential equation in Corollary \ref{HyperbolicDE} after substituting $z = 0$, and so, since $\rho(0)=0$, uniqueness gives that $\rho(q) = \Hh(0,q)$, which proves the desired formula for $\Hh(0,q)$.

Next, define $\chi(z,q) = \Hh(z,\omega(q))$. We have $\Hh(0,\omega(q)) = (\nu \circ \tau \circ \omega)(q) = \nu(q)$, and therefore by Corollary \ref{HyperbolicDE},
\begin{equation*}
(1-\alpha\beta\nu)\frac{\partial \Hh}{\partial q}(z,\omega(q)) = (z+\alpha+\beta)\chi+\omega.
\end{equation*}
But by the chain rule,
\begin{equation*} 
\frac{\partial\chi}{\partial q} = \frac{\partial \Hh}{\partial q}(z,\omega(q))\cdot(1-\alpha\beta\nu),
\end{equation*}
and so we get that
\begin{equation} \label{dchi}
\frac{\partial\chi}{\partial q} = (z+\alpha+\beta)\chi+\omega.
\end{equation}
Finally, define
\begin{equation} \label{DefR}
R(z,q) = (z+\alpha)(z+\beta)\chi(z,q)+z\omega(q)-\alpha\beta\nu(q).
\end{equation}
Using \eqref{domegadnu} and \eqref{dchi}, we get that
\begin{equation*}
\frac{\partial R}{\partial q} = (z+\alpha+\beta)R+z.
\end{equation*}
Since $R(z,0) = 0$, it follows that $R(z,q) = z\Phi_{z+\alpha+\beta}(q)$. Replacing $q$ with $\tau(q)$ in \eqref{DefR} therefore gives
\begin{equation*}
z\Phi_{z+\alpha+\beta}(\tau(q)) = (z+\alpha)(z+\beta)\Hh(z,q)+zq-\alpha\beta \Hh(0,q),
\end{equation*}
as desired.
\end{proof}
We get a particularly simple closed formula in the case of the Lorentz formal group law \eqref{LorentzFGL}.
\begin{cor} \label{LorentzGenFun}
In the case of the Lorentz formal group law, we have
\begin{equation*}
\Hh(z,q) = \frac{e^{\frac{z}{\alpha}\textnormal{arcsinh}(\alpha q)}-\sqrt{1+\alpha^2q^2}-zq}{z^2-\alpha^2}.
\end{equation*}
\end{cor}
\begin{proof}
By definition, namely \eqref{DefOmega} with $\beta = -\alpha$, we have
\begin{equation*}
\omega(q) = \frac{\alpha\Phi_{-\alpha}(q)+\alpha\Phi_{\alpha}(q)}{2\alpha} = \frac{e^{\alpha q}-e^{-\alpha q}}{2\alpha} = \frac{\text{sinh}(\alpha q)}{\alpha},
\end{equation*}
which has compositional inverse
\begin{equation} \label{tauarcsinh}
\tau(q) = \frac{\text{arcsinh}(\alpha q)}{\alpha}.
\end{equation}
By Theorem \ref{HyperbolicClosedFormula}, we have
\begin{equation*}
\Hh(0,q) = \frac{\Phi_{\alpha}(\tau)-\Phi_{-\alpha}(\tau)}{2\alpha} = \frac{\text{cosh}(\alpha\tau)-1}{\alpha^2},
\end{equation*}
and substituting \eqref{tauarcsinh} gives
\begin{equation} \label{H0qLorentz}
\Hh(0,q) = \frac{\sqrt{1+\alpha^2q^2}-1}{\alpha^2}.
\end{equation}
Substituting \eqref{tauarcsinh} and \eqref{H0qLorentz} into the formula for $\Hh(z,q)$ in Theorem \ref{HyperbolicClosedFormula} then yields the result.
\end{proof}
\begin{exmp}
Consider the case of $K$-theory, that is set $\alpha = 0$ in Theorem \ref{HyperbolicClosedFormula}. We have $\omega(q) = \Phi_{0}(q) = q$, which implies that $\tau(q) = q$. By Theorem \ref{HyperbolicClosedFormula}, we get
\begin{equation*}
\Hh(z,q) = \frac{e^{(z+\beta)q}-1-(z+\beta)q}{(z+\beta)^2},
\end{equation*}
which implies that
\begin{equation*}
\int_{\oM_{0,n}}^K \psi_1^j = \beta^{n-3-j}{n-3 \choose j},
\end{equation*}
agreeing with the general closed formula from \eqref{KClosedFormula}. In the Chow ring, that is by setting $\alpha = \beta = 0$, we get
\begin{equation*}
\Hh(z,q) = \frac{e^{zq}-1-zq}{z^2},
\end{equation*}
that is
\begin{equation*}
  \int_{\oM_{0,n}}^{\text{CH}} \psi_1^j =
  \begin{cases}
    1 , & j = n-3 , \\
    0 , & \text{otherwise} ,
  \end{cases}
\end{equation*}
as expected. 
\end{exmp}

\section{Localization} \label{LocalizationSection}

\subsection{Localization in cobordism}

In this section we give an alternative recursion for the cobordism classes $[\oM_{0,n}] \in \L^{*}$ using the technique of equivariant localization. Equivariant algebraic cobordism for smooth varieties was introduced by Deshpande \cite{Deshpande}, and later developed for schemes independently by Krishna \cite{Krishna2} and Heller--Malag\'{o}n-L\'{o}pez \cite{HML}. We restrict our attention in this section to the case of $\C^*$-actions, where over $k$ we understand $\C^*$ to mean $\G_{m}$. By \cite[Proposition 6.5]{Krishna2}, the $\C^*$-equivariant algebraic cobordism ring of $\Spec k$ is 
\begin{equation*}
\Omega_{\C^*}^{*}(\Spec k) \cong \L^*[\![t]\!],
\end{equation*}
where $t$ has degree 1, in particular is the equivariant first Chern class of the line bundle corresponding to the one-dimensional $\C^*$-representation of weight 1. Note that we have \cite{Krishna2} for two equivariant line bundles $L_{1},L_{2} \in \text{Pic}^{\C^*}(X)$,
\begin{equation*}
c_{1}^{\C^*}(L_{1} \otimes L_{2}) = c_{1}^{\C^*}(L_{1})+_{\widetilde\Omega} c_{1}^{\C^*}(L_{2}). 
\end{equation*}
Using the notation of Krishna \cite{Krishna1}, let $S = (\L^{*} \otimes_{\Z} \Q)[\![t]\!]$ and $S[M^{-1}]$ be the ring obtained from $S$ by inverting all elements of the form $mt$ for $m \in \Z \setminus \{0\}$. Let $X$ be a smooth projective variety with $\C^{*}$-action, and denote by $X^{\C^{*}}$ the fixed-point locus $\coprod F$. Here the component $F$ has codimension $d_{F}$ in $X$, and $i_{F} \colon F \to X$ is the inclusion with normal bundle $N_{F}X$. Then, the localization formula of \cite[Corollary 7.3]{Krishna1} says that for any $\alpha \in \Omega_{\C^*}^{*}(X;\Q) \otimes_{S} S[M^{-1}]$, where $\Omega_{\C^*}^{*}(X;\Q)$ is the $\C^*$-equivariant algebraic cobordism ring of $X$ with rational coefficients, we have
\begin{equation*}
\alpha = \sum_{F} (i_{F})_{*}\frac{i_{F}^{*}(\alpha)}{c_{d_F}^{\C^{*}}(N_{F}X)}.
\end{equation*}
Pushing forward to a point in equivariant cobordism gives
\begin{equation} \label{CobordismLocalizationFormula}
\int_{X}^{\Omega_{\C^*}} \alpha = \sum_{F} \int_{F}^{\Omega_{\C^*}} \frac{i_{F}^{*}(\alpha)}{c_{d_F}^{\C^{*}}(N_{F}X)}.
\end{equation}

Inspired by ideas of Faber--Pandharipande \cite{FaberPandharipande} in higher genus and ordinary cohomology, we will apply cobordism localization to $\oM_{0,n}(\P^1,1)$, the moduli space of genus zero stable maps to $\P^1$ of degree one, with $n$ marked points. Note that $\oM_{0,n}(\P^1,1)$ is isomorphic to $\P^1[n]$, the Fulton--MacPherson compactification \cite{FultonMacPherson} of the configuration space of $n$ points in $\P^1$. Consider the $\C^{*}$-action on $\P^1$ given by $s \cdot [x,y] = [x,sy]$. There are precisely two fixed points,
\begin{equation*}
p_{0} = [1,0] \hspace{0.5cm} \text{and} \hspace{0.5cm} p_{\infty} = [0,1].
\end{equation*}
Note that the action of $\C^*$ on the tangent spaces $T_{p_0}\P^1$ and $T_{p_\infty}\P^1$ have weights $1$ and $-1$ respectively. Now, the action of $\C^*$ on $\P^1$ induces an action on $\oM_{0,n}(\P^1,1)$ by translation of maps, i.e. for a stable map $f$ we have $s \cdot[f] = [s \cdot f]$. As above, let $t$ be the equivariant first Chern class of the line bundle corresponding to the $\C^*$-action on $T_{p_0}\P^1$. For each $n \ge 3$, define
\begin{equation} \label{AnDef}
A_{n}(t) = \int_{\oM_{0,n}}^\Omega \frac{1}{t+_{\Omega}\psi_1},
\end{equation}
where in this pushforward we treat $t$ as a formal variable, and form the exponential generating function
\begin{equation} \label{EDef}
E(t,q) = t-q+\sum_{n \ge 3} \frac{q^{n-1}}{(n-1)!}A_{n}(t).
\end{equation}
\begin{thm} \label{CobordismLocalization}
The values of $\int_{\oM_{0,n}}^\Omega \psi_1^{j}$ for all $0 \le j \le n-3$ are uniquely determined by the condition that the coefficients of all negative powers of $t$ (coefficient-wise in $q$) in the expression
\begin{equation*}
\frac{\partial E(t,q)}{\partial q}\cdot \frac{E\bigl(\overline{t},q\bigr)}{\overline{t}}
\end{equation*}
are equal to zero, together with the fact that $\int_{\oM_{0,3}}^\Omega 1 = 1$. 
\end{thm}
\begin{proof}
Consider the hyperplane class $h \in \Omega^{1}(\P^{1})$, namely $h = [\text{pt} \to \P^1]$ where $\text{pt} = \Spec k$, here using the geometric description of algebraic cobordism \cite{LevinePandharipande}. We lift $h$ to an equivariant class $h_{0} \in \Omega_{\C^*}^{1}(\P^1)$ as follows. We choose a linearization of $\Oo_{\P^1}(1)$ where the weights over the fixed points $p_{0}$ and $p_{\infty}$ are $1$ and $0$ respectively, and define $h_{0} = c_{1}^{\C^*}(\Oo_{\P^1}(1))$. That is, we have
\begin{equation*}
h_{0} \vert_{p_0} = t \hspace{0.5cm} \text{and} \hspace{0.5cm} h_{0} \vert_{p_\infty} = 0.
\end{equation*}
We will apply the localization formula \eqref{CobordismLocalizationFormula} to the equivariant class $\text{ev}_{1}^{*}(h_0) \in \Omega_{\C^*}^{1}(\oM_{0,n}(\P^1,1))$, where $\text{ev}_{1} \colon \oM_{0,n}(\P^1,1) \to \P^1$ is the evaluation map at the marked point labeled by $1$. Note that we have \cite[Section 2.2]{Krishna1}
\begin{equation*}
\int_{\oM_{0,n}(\P^1,1)}^\Omega \text{ev}_{1}^{*}(h) = \left(\int_{\oM_{0,n}(\P^1,1)}^{\Omega_{\C^*}} \text{ev}_{1}^{*}(h_0)\right) \bigg\vert_{t = 0},
\end{equation*}
where the left-hand side is the non-equivariant pushforward, and the right-hand side is the result of taking the equivariant pushforward, which is an element of $\L^{*}[\![t]\!]$ and hence has no negative powers of $t$, and substituting $t = 0$. 

In order to apply the localization formula, we need to identify the fixed loci and the top Chern classes of their normal bundles. Following the discussion in Okounkov--Pandharipande \cite{OkounkovPandharipande}, we know that the fixed locus is the disjoint union of components
\begin{equation*}
\oM_{0,n}(\P^1,1)^{\C^*} = \coprod_{T \subset \{1,\dots,n\}} \oM_{0,\vert{T}\vert+1} \times \oM_{0,\vert{T^{c}}\vert+1},
\end{equation*}
with the convention that if $\vert{T}\vert \le 1$ or $\vert{T^c}\vert \le 1$, then the corresponding factor in the product is taken to be a point. Here $\oM_{0,\vert{T}\vert+1} \times \oM_{0,\vert{T^{c}}\vert+1}$ corresponds to the locus of stable maps where the domain curve has a unique component mapping isomorphically to $\P^1$, with two (possibly reducible) contracted components containing markings labeled by $T$ and $T^c$, with the two nodes lying over $p_0$ and $p_\infty$ respectively. If $\vert{T}\vert = 1$ or $\vert{T^c}\vert = 1$, then there is no corresponding contracted component, and the marked point on the non-contracted component maps to $p_0$ or $p_\infty$ respectively. 

Equivalently, one can view the fixed-point locus as being indexed by graphs with two vertices (labeled by $p_0$ and $p_\infty$) and one edge, with half-edges (markings) incident to $p_0$ labeled by elements of $T$ and half-edges incident to $p_{\infty}$ labeled by elements of $T^c$. We call a vertex \textit{stable} if the number of half edges incident to it is at least 2. 

Following the same deformation sequence argument as in \cite[Section 6.3]{OkounkovPandharipande} but replacing the additive formal group law with the universal formal group law when taking first Chern classes of tensor products and duals of line bundles, we have for each component $F$ of the fixed locus,
\begin{equation*}
\frac{1}{c_{d_F}^{\C^*}(N_{F}\oM_{0,n}(\P^1,1))} = \frac{1}{N(p_0)} \cdot \frac{1}{N(p_{\infty})}\cdot \frac{1}{N(e)},
\end{equation*}
where the vertex and edge contributions $N(p_0)$, $N(p_{\infty})$ and $N(e)$ are given by the following:
\begin{itemize}
\item If $p_0$ is a stable vertex, then
\begin{equation*}
\frac{1}{N(p_0)} = \frac{1}{t+_{\widetilde\Omega}(-\psi_1)}.
\end{equation*}
If $p_\infty$ is a stable vertex, then
\begin{equation*}
\frac{1}{N(p_\infty)} = \frac{1}{\iota(t)+_{\widetilde\Omega}(-\psi_1)}.
\end{equation*}
\item If $p_0$ or $p_\infty$ has exactly one marking, then in either case $i \in \{0,\infty\}$ we have
\begin{equation*}
\frac{1}{N(p_i)} = 1.
\end{equation*}
\item If $p_0$ or $p_\infty$ have no markings, then
\begin{equation*}
\frac{1}{N(p_0)} = t \hspace{0.5cm} \text{or} \hspace{0.5cm} \frac{1}{N(p_\infty)} = \iota(t),
\end{equation*}
respectively. 
\item We have
\begin{equation*}
\frac{1}{N(e)} = \frac{1}{t\cdot\iota(t)}.
\end{equation*}
\end{itemize}
Define
\begin{equation*}
\widetilde{A}_{n}(t) = \int_{\oM_{0,n}}^{\Omega} \frac{1}{t+_{\widetilde\Omega} (-\psi_1)} \hspace{0.5cm} \text{and} \hspace{0.5cm} \widetilde{E}(t,q) = t+q+\sum_{n \ge 3} \frac{q^{n-1}}{(n-1)!}\widetilde{A}_{n}(t) .
\end{equation*}
By the localization formula \eqref{CobordismLocalizationFormula}, it follows that
\begin{equation} \label{LocalizationM0n}
\int_{\oM_{0,n}(\P^1,1)}^{\Omega_{\C^*}} \text{ev}_{1}^{*}(h_0) = \widetilde{A}_{n+1}(t)+\frac{\widetilde{A}_{n}(\iota(t))}{\iota(t)}+\frac{n-1}{\iota(t)}\widetilde{A}_{n}(t)+\frac{1}{\iota(t)}\sum_{j=2}^{n-2}{n-1 \choose j} \widetilde{A}_{n-j+1}(t)\widetilde{A}_{j+1}(\iota(t)).
\end{equation}
Suppose we know inductively all one-point psi-class intersections on $\oM_{0,m}$ for $m \le n$, where $\int_{\oM_{0,3}}^\Omega 1 = 1$. Then, we know $\widetilde{A}_{m}(t)$ for all $m \le n$, where $\widetilde{A}_{3}(t) = 1/t$. Since the coefficients of all negative powers of $t$ in the sum on the right-hand side of \eqref{LocalizationM0n} are equal to zero, it follows that we know the coefficients of all negative powers of $t$ in $\widetilde{A}_{n+1}(t)$. 

This determines all one-point psi-class intersections on $\oM_{0,n+1}$ as follows. Writing $x+_{\widetilde\Omega} y = x+y-xyQ(x,y)$ where $Q \in 
\L^*[\![x,y]\!]$, we have that
\begin{align*}
\widetilde{A}_{n+1}(t) &= \int_{\oM_{0,n+1}}^\Omega \frac{1}{t-\psi_1(1-tQ(t,-\psi_1))} \\
&= \sum_{j=0}^{n-2}\int_{\oM_{0,n+1}}^\Omega \frac{\psi_1^{j}}{t^{j+1}}(1-tQ(t,-\psi_1))^{j}. \\
\end{align*}
Since $\psi_{1}^{n-1} = 0$, the coefficient of $t^{-(n-1)}$ is $\int_{\oM_{0,n+1}}^{\Omega} \psi_1^{n-2}$. The coefficient of $t^{-(j+1)}$ for $0 \le j < n-2$ is of the form
\begin{equation*}
\int_{\oM_{0,n+1}}^{\Omega} \psi_1^j+b_{1}\int_{\oM_{0,n+1}}^{\Omega} \psi_1^{j+1}+\dots+b_{n-2-j}\int_{\oM_{0,n+1}}^{\Omega} \psi_1^{n-2},
\end{equation*}
where the constants $b_{1},\dots,b_{n-2-j} \in \L^*$ are determined by the power series $Q(x,y)$, and hence we know the value of $\int_{\oM_{0,n+1}}^{\Omega} \psi_1^j$ by downward induction. The right-hand side of \eqref{LocalizationM0n} is precisely the coefficient of $q^{n-1}/(n-1)!$ in 
\begin{equation} \label{EquationFirstOne}
\frac{\partial \widetilde{E}(t,q)}{\partial q}\cdot \frac{\widetilde{E}\bigl(\iota(t),q\bigr)}{\iota(t)}.
\end{equation}
By \eqref{UsualFGL}, we have that $E(t,q) = -\widetilde{E}(-t,q)$, and $\overline{t} = -\iota(-t)$. It follows that
\begin{equation} \label{EquationSecondOne}
\frac{\partial E(t,q)}{\partial q}\cdot \frac{E\bigl(\overline{t},q\bigr)}{\overline{t}} = -\frac{\partial \widetilde{E}(-t,q)}{\partial q}\cdot \frac{\widetilde{E}\bigl(\iota(-t),q\bigr)}{\iota(-t)},
\end{equation}
which implies the result, as the right-hand side of \eqref{EquationSecondOne} has no negative powers of $t$ if and only if \eqref{EquationFirstOne} has no negative powers of $t$. 
\end{proof}

\subsection{Localization in the hyperbolic formal group law} 

In this section we give an alternative proof of the closed formula in Theorem \ref{HyperbolicClosedFormula} using Theorem \ref{CobordismLocalization} restricted to the hyperbolic formal group law. 

\begin{proof}[Proof of Theorem \ref{HyperbolicClosedFormula}]
Let $s = -1/t$. By \eqref{InverseHyperbolicFGL}, we have
\begin{equation*}
-\frac{1}{\overline{t}} = \mu_1-s.
\end{equation*}
Defining $\E(s,q) = -sE(-1/s,q)$, we have by Theorem \ref{CobordismLocalization} that
\begin{equation*}
\frac{\partial E(t,q)}{\partial q}\cdot \frac{E(\overline{t},q)}{\overline{t}} = -\frac{1}{s}\frac{\partial \E(s,q)}{\partial q}\cdot \E(\mu_{1}-s,q)
\end{equation*}
has no positive powers of $s$, coefficient-wise in $q$. Moreover, note that
\begin{equation} \label{An1/s}
-A_{n}\biggl(-\frac{1}{s}\biggr) = \int_{\oM_{0,n}}^H \frac{s-\mu_2 \psi_1}{1-\psi_1(s-\mu_1)}
\end{equation}
is polynomial in $s$. Therefore, it follows that
\begin{equation} \label{ProductRule1}
\frac{\partial \E(s,q)}{\partial q}\cdot \E(\mu_{1}-s,q) = sC(q),
\end{equation}
where $C(q)$ is a power series in $q$ over $\Q[\mu_1,\mu_2]$ with constant term 1, in particular it is independent of $s$. Replacing $s$ with $\mu_1-s$ in \eqref{ProductRule1} gives
\begin{equation} \label{ProductRule2}
\frac{\partial \E(\mu_1-s,q)}{\partial q}\cdot \E(s,q) = (\mu_1-s)C(q),
\end{equation}
and hence by adding \eqref{ProductRule1} and \eqref{ProductRule2}, we have
\begin{equation*}
\frac{\partial}{\partial q}(\E(s,q) \cdot \E(\mu_1-s,q)) = \mu_1 C(q). 
\end{equation*}
Since $\E(s,0) \cdot \E(\mu_1-s,0) = 1$, it follows that $\E(s,q) \cdot \E(\mu_1-s,q) = B(q)$ for some power series $B(q)$ over $\Q[\mu_1,\mu_2]$ with constant term $1$. Dividing \eqref{ProductRule1} by this yields
\begin{equation*}
\frac{\partial \E(s,q)}{\partial q} = s\lambda(q)\cdot\E(s,q),
\end{equation*}
where $\lambda(q) = C(q)/B(q)$. Let $\tau(q)$ be a power series with $\tau(0) = 0$ such that $\tau'(q) = \lambda(q)$. Then, since $\E(s,0) = 1$, we have that
\begin{equation*}
\E(s,q) = e^{s\tau(q)},
\end{equation*}
and hence
\begin{equation} \label{ECollapse}
E(t,q) = te^{\frac{-\tau(q)}{t}}.
\end{equation}
We now use that $\mu_1 = \alpha+\beta$ and $\mu_2 = -\alpha\beta$. Observe that if we set $s = \alpha$ in \eqref{An1/s}, we get
\begin{equation*}
-A_{n}\biggl(-\frac{1}{\alpha}\biggr) = \int_{\oM_{0,n}}^H \frac{\alpha+\alpha\beta\psi_1}{1+\beta\psi_1} = \alpha\int_{\oM_{0,n}}^{H} 1.
\end{equation*}
Therefore, by \eqref{EDef}, it follows that
\begin{equation*}
E\biggl(-\frac{1}{\alpha},q\biggr) = -\frac{1}{\alpha}-q-\alpha \Hh(0,q).
\end{equation*}
That is, using \eqref{ECollapse}, we have
\begin{equation*}
\Phi_{\alpha}(\tau(q)) =  q+\alpha\Hh(0,q).
\end{equation*}
Similarly, we also have
\begin{equation*}
\Phi_{\beta}(\tau(q)) =  q+\beta\Hh(0,q),
\end{equation*}
which together imply that
\begin{equation*}
q = \frac{\alpha\Phi_{\beta}(\tau(q))-\beta\Phi_{\alpha}(\tau(q))}{\alpha-\beta} \hspace{0.5cm} \text{and} \hspace{0.5cm} \Hh(0,q) = \frac{\Phi_{\alpha}(\tau(q))-\Phi_{\beta}(\tau(q))}{\alpha-\beta}.
\end{equation*}
Let $z = -\alpha-\beta-1/t$. By \eqref{An1/s}, we have
\begin{equation*}
A_{n}(t) = \frac{1}{t}\int_{\oM_{0,n}}^H \frac{1}{1-z\psi_1}-\alpha\beta\int_{\oM_{0,n}}^H \frac{\psi_1}{1-z\psi_1}.
\end{equation*}
Using \eqref{EDef}, it follows that
\begin{equation*}
E(t,q) = t-q+\frac{1}{t}\Hh(z,q)-\alpha\beta\frac{\Hh(z,q)-\Hh(0,q)}{z}. 
\end{equation*}
Re-arranging gives
\begin{align*}
-z(E(t,q)-t+q)+\alpha\beta\Hh(0,q) &= \biggl(-\frac{z}{t}+\alpha\beta\biggr)\Hh(z,q) \\
&= (z+\alpha)(z+\beta)\Hh(z,q). \\
\end{align*}
By \eqref{ECollapse}, we get that
\begin{equation*}
E(t,q)-t = -\Phi_{z+\alpha+\beta}(\tau(q)),
\end{equation*}
and hence
\begin{equation*}
\Hh(z,q) = \frac{z(\Phi_{z+\alpha+\beta}(\tau(q))-q)+\alpha\beta\Hh(0,q)}{(z+\alpha)(z+\beta)}. 
\end{equation*}
\end{proof}

\section{Twisted $\chi_y$-genera}

\subsection{The $\chi_y$-genus}
We now work over $\C$. Recall \cite{Hirzebruch} that for a smooth proper variety $X$, we have the Hirzebruch $\chi_y$-genus 
\begin{equation*}
\chi_y(X) = \sum_{p} \chi(X,\Omega_X^p)y^p,
\end{equation*}
where $\Omega_X^p$ is the sheaf of algebraic $p$-forms on $X$. We have using Hodge theory that 
\begin{equation*}
\chi_y(X) = \sum_{p,q} (-1)^q \dim_{\C} H^{q}(X,\Omega_X^p)y^p = \sum_{p,q} (-1)^q h^{p,q}(X)y^p.
\end{equation*}
By the Hirzebruch--Riemann--Roch theorem, the (unnormalized) formal power series corresponding to the $\chi_y$-genus is
\begin{equation*}
\overline{Q}_{\chi_y}(u) = \frac{u(1+ye^{-u})}{1-e^{-u}},                                          
\end{equation*}
since if $\lambda_1,\dots,\lambda_{\dim_{k}X}$ are the Chern roots of $T_X$, then
\begin{equation*}
\chi_y(X) = \int_{X}^{\text{CH}} \prod_i \frac{\lambda_i(1+ye^{-\lambda_i})}{1-e^{-\lambda_i}}.
\end{equation*}
We now normalize the series $\overline{Q}_{\chi_y}(u)$ so that it has constant term 1, that is we define
\begin{equation*}
Q_{\chi_y}(u) = \frac{\overline{Q}_{\chi_y}((1+y)u)}{1+y} = \frac{u(1+ye^{-(1+y)u})}{1-e^{-(1+y)u}}.
\end{equation*}
Note that we also have
\begin{equation*}
\chi_y(X) = \int_{X}^{\text{CH}} \prod_{i} \frac{\lambda_i(1+ye^{-(1+y)\lambda_i})}{1-e^{-(1+y)\lambda_i}}.
\end{equation*}
In the convention of \eqref{UsualFGL}, the formal group law $u+_{\widetilde{\chi_y}}v$ corresponding to the $\chi_y$-genus has exponential (which is the compositional inverse of the logarithm $\ell_{\widetilde{\chi_y}}(u)$ as defined for cobordism in \eqref{Logarithm}) 
\begin{equation*}
\exp_{\widetilde{\chi_y}}(u) = \frac{u}{Q_{\chi_y}(u)} = \frac{1-e^{-(1+y)u}}{1+ye^{-(1+y)u}}.
\end{equation*}
It follows then from a direct calculation that 
\begin{equation*}
u+_{\widetilde{\chi_y}} v = \exp_{\widetilde{\chi_y}}(\ell_{\widetilde{\chi_y}}(u)+\ell_{\widetilde{\chi_y}}(v)) = \frac{u+v-(1-y)uv}{1+yuv},
\end{equation*}
and 
\begin{equation*}
u+_{\chi_y} v = -((-u)+_{\widetilde{\chi_y}} (-v)) = \frac{u+v+(1-y)uv}{1+yuv}. 
\end{equation*}
Hence, we can obtain $u+_{\chi_y} v$ from the hyperbolic formal group law (see \eqref{HyperbolicFGL}) by setting $\mu_1 = 1-y$ and $\mu_2 = y$, or $\alpha = 1$ and $\beta = -y$. 
\begin{rem} \label{OrientedRemark}
In order for $u+_{\chi_y} v$ to come from an oriented cohomology theory, we need to keep the formal variable $\beta$ of degree $-1$, so that
\begin{equation*}
u+_{\chi_y} v = \frac{u+v+(1-y)\beta uv}{1+y\beta^2uv}.
\end{equation*}
\end{rem}

\subsection{Poincar\'{e} polynomials of $\oM_{0,n}$}
We now explain how to recover known generating function formulas for the Poincar\'{e} polynomials of $\oM_{0,n}$, in particular we give a new proof of formulas of Getzler and Manin \cite{Getzler,Manin}. Define
\begin{equation*}
\text{P}_{n}(z) = \sum_{j = 0}^{n-3}\dim_{\C} H^{2j}(\oM_{0,n},\C)z^j,
\end{equation*}
and following the convention in \cite{Cotti}, define
\begin{equation*}
\F(z,q) = q+\sum_{n \ge 3} \frac{q^{n-1}}{(n-1)!} \text{P}_{n}(z).
\end{equation*}
\begin{cor}[\cite{Getzler,Manin}] \label{GetzlerManin}
The series $\F(z,q)$ is uniquely determined by the differential equation
\begin{equation} \label{GetzlerManin1}
\frac{\partial \F(z,q)}{\partial q} = \frac{1+\F}{1+zq-z \F},
\end{equation}
together with the initial condition $\F(z,0) = 0$. Equivalently, it is the unique power series in $q$ with coefficients polynomial in $z$ and $\F(z,0) = 0$ satisfying
\begin{equation} \label{GetzlerManin2}
(1+\F)^{z} = z^2\F-z(z-1)q+1.
\end{equation}
\end{cor}
\begin{proof}
By Keel \cite{Keel}, we know that the Hodge structure on the cohomology of $\oM_{0,n}$ is of pure Tate type. That is, we have $h^{p,q}(\oM_{0,n}) = 0$ for $p \neq q$, and so
\begin{equation*}
\chi_y(\oM_{0,n}) = \sum_{p,q} (-1)^qh^{p,q}(\oM_{0,n})y^p = \sum_{p} (-1)^p \dim_{\C}H^{2p}(\oM_{0,n},\C)y^p.
\end{equation*}
Hence, we have $\text{P}_{n}(z) = \chi_{-z}(\oM_{0,n})$, and so
\begin{equation*}
\F(z,q) = q+\sum_{n \ge 3} \frac{q^{n-1}}{(n-1)!}\int_{\oM_{0,n}}^{\chi_{-z}} 1.
\end{equation*}
Here these pushforwards are calculated in the associated oriented cohomology in the sense of Remark \ref{OrientedRemark} with $\beta = 1$ set afterwards. Using the differential equation of Corollary \ref{HyperbolicDE} in the case $z = 0$, and then replacing $\Hh(0,q)$ with $\F(z,q)-q$ and setting $\mu_1 = 1+z$ and $\mu_2 = -z$ gives
\begin{equation*}
(1-z(\F-q))\left(\frac{\partial \F(z,q)}{\partial q}-1\right) = (1+z)(\F-q)+q. 
\end{equation*}
Re-arranging yields \eqref{GetzlerManin1}. Next, we use our closed formula from Theorem \ref{HyperbolicClosedFormula}. Substituting $\alpha = 1$ and $\beta = z$ into \eqref{DefOmega} gives 
\begin{equation*}
\omega(q) = \frac{\Phi_{z}(q)-z\Phi_1(q)}{1-z} = \frac{e^{zq}-1-z^2(e^q-1)}{z(1-z)}. 
\end{equation*}
Since $\tau(q)$ is by definition the compositional inverse of $\omega(q)$, we have 
\begin{equation} \label{equationone}
q = \frac{e^{z\tau}-1-z^2(e^\tau-1)}{z(1-z)}.
\end{equation}
However, by the $z = 0$ case of Theorem \ref{HyperbolicClosedFormula}, we have
\begin{equation} \label{equationtwo}
\F-q = \frac{z(e^{\tau}-1)-(e^{z\tau}-1)}{z(1-z)}.
\end{equation}
Adding \eqref{equationone} and \eqref{equationtwo} gives $\F = e^{\tau}-1$. Substituting this into \eqref{equationtwo} then gives
\begin{equation*}
\F-q = \frac{z\F-(1+\F)^{z}+1}{z(1-z)},
\end{equation*}
which after re-arranging yields the result. 
\end{proof}

\begin{rem}
The quadratic recurrence for the Poincar\'{e} polynomials of $\oM_{0,n}$, due to Keel \cite{Keel} and Manin \cite{Manin},
\begin{equation*}
\text{P}_{n+1}(z) = (1+z)\text{P}_{n}(z)+z\sum_{\substack{i+j = n \\ i,j \ge 2}} {n-1 \choose i} \text{P}_{i+1}(z)\text{P}_{j+1}(z),
\end{equation*}
follows directly from the string equation as given in Proposition \ref{NewStringEquation}, using that $\text{P}_{n}(z) = \chi_{-z}(\oM_{0,n})$. 
\end{rem}

\subsection{Twisted $\chi_y$-genera of $\oM_{0,n}$}
For a smooth proper variety $X$ over $\C$ and a holomorphic line bundle $L$ on $X$, we consider \cite{Hirzebruch} the $\chi_y$ genus twisted by $L$
\begin{equation*}
\chi_y(X,L) := \sum_{p} \chi(X,\Omega_{X}^{p} \otimes L)y^p = \sum_{p,q} (-1)^q \dim_{\C} H^{q}(X,\Omega_{X}^p \otimes L)y^p.
\end{equation*}
We will consider the case $X = \oM_{0,n}$ and $L = L_{1}^{j}$, that is we look at tensor powers of a single tautological line bundle on $\oM_{0,n}$. Form the exponential generating function
\begin{equation*}
G(z,q) = \sum_{n \ge 3} \frac{q^{n-1}}{(n-1)!}\sum_{j \ge 0} z^{j} \chi_y(\oM_{0,n},L_1^j).
\end{equation*}
As in Section \ref{HyperbolicFGLSection} and the proof of Corollary \ref{GetzlerManin}, recall for any parameter $\lambda$ that
\begin{equation*}
\Phi_{\lambda}(q) = \sum_{j \ge 1} \lambda^{j-1}\frac{q^j}{j!} = \frac{e^{\lambda q}-1}{\lambda},
\end{equation*}
where $\Phi_0(q) = q$, and let $\tau(q)$ be the compositional inverse of the series
\begin{equation*}
\omega(q) = \frac{\Phi_{-y}(q)+y\Phi_{1}(q)}{1+y}.
\end{equation*}
\begin{prop} \label{TwistedChiy}
The generating function $G(z,q)$ is given by
\begin{equation*}
G(z,q) = \frac{\Phi_{\frac{1+yz}{1-z}}(\tau(q))-\Phi_{-y}(\tau(q))}{1+y}.
\end{equation*}
\end{prop}
\begin{proof}
Let $\ell = c_{1}^{\text{CH}}(L_1)$ be the ordinary first Chern class of $L_1$, and let $\lambda_1,\dots,\lambda_{n-3}$ be the ordinary Chern roots of $T_{\oM_{0,n}}$. By the Hirzebruch--Riemann--Roch theorem for the $\chi_y$-characteristic \cite{Yokura}, we have
\begin{equation*}
\chi_{y}(\oM_{0,n},L_{1}^j) = \int_{\oM_{0,n}}^{\text{CH}} e^{j(1+y)\ell} \prod_{i} \frac{\lambda_{i}(1+ye^{-(1+y)\lambda_i})}{1-e^{-(1+y)\lambda_i}}. 
\end{equation*}
The generalized Chern character $\text{ch}^{\chi_y}$ sends $c_{1}^{\chi_y}(L_1)$ to $\text{exp}_{\widetilde{\chi_y}}(\ell)$, and hence
\begin{equation*}
\text{ch}^{\chi_y}(\psi_1) = -\text{exp}_{\widetilde{\chi_y}}(-\ell) = \frac{e^{(1+y)\ell}-1}{1+ye^{(1+y)\ell}},
\end{equation*}
and
\begin{equation*}
\text{ch}^{\chi_y}\biggl(\frac{1+\psi_1}{1-y\psi_1}\biggr) = e^{(1+y)\ell}. 
\end{equation*}
Therefore, in the oriented cohomology theory associated to the $\chi_y$-genus, we have
\begin{equation*}
\chi_{y}(\oM_{0,n},L_{1}^j) = \int_{\oM_{0,n}}^{\chi_y} \biggl(\frac{1+\psi_1}{1-y\psi_1}\biggr)^j.
\end{equation*}
We then have 
\begin{align*}
G(z,q) &= \sum_{n \ge 3} \frac{q^{n-1}}{(n-1)!} \sum_{j \ge 0} \int_{\oM_{0,n}}^{\chi_y} z^j \biggl(\frac{1+\psi_1}{1-y\psi_1}\biggr)^{j} \\
&= \sum_{n \ge 3} \frac{q^{n-1}}{(n-1)!} \int_{\oM_{0,n}}^{\chi_y} \frac{1-y\psi_1}{1-z-(y+z)\psi_1}. \\
\end{align*}
Let $\widetilde{z} = (y+z)/(1-z)$. Then, we have
\begin{equation*}
\frac{1-y\psi_1}{1-z-(y+z)\psi_1} = \frac{1}{1-z} \cdot \frac{1-y\psi_1}{1-\widetilde{z}\psi_1} = \frac{1}{1-z}\biggl(1+(\widetilde{z}-y)\sum_{j \ge 1} \widetilde{z}^{j-1}\psi_1^j\biggr),
\end{equation*}
and so comparing with the generating function in Section \ref{HyperbolicFGLSection},
\begin{align*}
G(z,q) = \frac{1}{1-z}\biggl(\Hh(0,q)+\frac{\widetilde{z}-y}{\widetilde{z}}(\Hh(\widetilde{z},q)-\Hh(0,q))\biggr). 
\end{align*}
By Theorem \ref{HyperbolicClosedFormula} where we set $\alpha = 1$ and $\beta = -y$, we have
\begin{equation*}
\Hh(z,q) = \frac{z(\Phi_{z+1-y}(\tau(q))-q)-y\Hh(0,q)}{(z+1)(z-y)}.
\end{equation*}
A direct calculation gives
\begin{equation*}
G(z,q) = \frac{\Phi_{\frac{1+yz}{1-z}}(\tau(q))-q+y\Hh(0,q)}{1+y}.
\end{equation*}
But again by Theorem \ref{HyperbolicClosedFormula}, we have
\begin{equation*}
\Hh(0,q) = \frac{\Phi_{1}(\tau(q))-\Phi_{-y}(\tau(q))}{1+y},
\end{equation*}
and so since $\tau(q)$ is the compositional inverse of $\omega(q)$,
\begin{align*}
-q+y\Hh(0,q) &= -\biggl(\frac{\Phi_{-y}(\tau(q))+y\Phi_{1}(\tau(q))}{1+y}\biggr)+y\biggl(\frac{\Phi_{1}(\tau(q))-\Phi_{-y}(\tau(q))}{1+y}\biggr) \\
&= -\Phi_{-y}(\tau(q)),
\end{align*}
which gives the result. 
\end{proof}

\begin{cor}
Setting $y = 1$, the generating function for the twisted signatures is given by
\begin{equation*}
G(z,q) = \frac{1}{2}\left(\frac{1-z}{1+z}(e^{\frac{1+z}{1-z}\textnormal{arcsinh}\,q}-1)+e^{-\textnormal{arcsinh}\,q}-1\right).
\end{equation*}
\end{cor}
\begin{proof}
As in the proof of Corollary \ref{LorentzGenFun}, where $\alpha = 1$ in this case, we have
\begin{equation*}
\tau(q) = \text{arcsinh}(q).
\end{equation*}
The result then follows directly by substituting $y = 1$ in Proposition \ref{TwistedChiy}. 
\end{proof}
\begin{rem}
Substituting $y = 0$ in Proposition \ref{TwistedChiy} recovers Lee's formulas \cite{Lee} for the Euler characteristics $\chi(\oM_{0,n},L_{1}^j)$, namely
\begin{equation*}
G(z,q) = (1-z)\bigl(e^{\frac{q}{1-z}}-1\bigr)-q,
\end{equation*}
which gives
\begin{equation*}
\chi(\oM_{0,n},L_{1}^j) = {j+n-3 \choose j}.
\end{equation*}
\end{rem}
\FloatBarrier

\end{document}